\documentclass{amsart}
\usepackage{amsmath,amsthm,amssymb,bm,mathtools,fixmath}
\usepackage{pifont}
\usepackage{url}
\numberwithin{equation}{section}
\usepackage{xcolor, color}
\usepackage[table]{xcolor}
\usepackage{colortbl}
 \usepackage{indentfirst}
\usepackage{graphicx}
\usepackage{booktabs}
\usepackage{multirow}
\usepackage{tikz}
\usetikzlibrary{decorations.pathreplacing, decorations.pathmorphing, decorations.shapes}
\usetikzlibrary{arrows,calc}
\usepackage{mathrsfs,esint}
\usepackage{comment}
\usepackage{enumerate}
\usepackage{hyperref}    
\hypersetup{
     colorlinks = true,
     linkcolor = red,
     anchorcolor =green!50!black,
     citecolor = green!50!black,
     filecolor =green!50!black,
     urlcolor = green!50!black
}
\usepackage{cleveref} 
\usepackage{enumitem}
\usepackage{soul}

\usepackage[top=1in, bottom=1in, left=1.25in, right=1.25in]{geometry}

\usepackage{float}
\usepackage[ruled, vlined]{algorithm2e}

\usepackage{subcaption}
\usepackage{pgfplots}
\pgfplotsset{compat=newest, compat/show suggested version=false}
\usepackage{tikz}
\usetikzlibrary{intersections}
\usepackage{tikz-3dplot}
\tdplotsetmaincoords{60}{20}

\newtheorem{theorem}{Theorem}
\newtheorem{lemma}[theorem]{Lemma} 
 
\newtheorem{remark}[theorem]{Remark}

\newtheorem{assumption}[theorem]{Assumption}

\newcommand{\de}{\mathop{}\!\mathrm{d}}

\DeclareMathOperator{\tr}{tr}
\DeclareMathOperator{\supp}{supp}

\DeclarePairedDelimiter{\norm}{\lVert}{\rVert}

\newcommand{\Lap}{\upDelta}

\newcommand{\beq}{\begin{equation}}
\newcommand{\eeq}{\end{equation}}

\def\cA{\mathcal{A}}
\def\cB{\mathcal{B}}

\def\cE{\mathcal{E}}
\def\cF{\mathcal{F}}

\def\cK{\mathcal{K}}

\def\cR{\mathcal{R}}

\def\cV{\mathcal{V}}

\def\N{\mathbb{N}}

\def\R{\mathbb{R}}

\def\sD{\mathscr{D}}

\usepackage{todonotes}

\title[]{
Sparse RBF Networks for PDEs and nonlocal equations: function space theory, operator calculus, and training algorithms
}

\author{Zihan Shao}
\address{Department of Mathematics, University of California, San Diego, CA 92093, United States} 
\email{z6shao@ucsd.edu}
\author{Konstantin Pieper}
\address{Computer Science and Mathematics Division, Oak Ridge National Laboratory, Oak Ridge, TN 37831, United States}
\email{pieperk@ornl.gov}
\author{Xiaochuan Tian}
\address{Department of Mathematics, University of California, San Diego, CA 92093, United States} 
\email{xctian@ucsd.edu}
\date{\today}

\begin{document}

\begin{abstract}
This work presents a systematic analysis and extension of the sparse radial basis function network (SparseRBFnet) previously introduced for solving nonlinear partial differential equations (PDEs). Based on its adaptive-width shallow kernel network formulation, we further investigate its function-space characterization, operator evaluation, and computational algorithm. We provide a unified description of the solution space for a broad class of radial basis functions (RBFs). Under mild assumptions, this space admits a characterization as a Besov space, independent of the specific kernel choice. We further demonstrate how the explicit kernel-based structure enables quasi-analytical evaluation of both differential and nonlocal operators, including fractional Laplacians. On the computational end, we study the adaptive-width network and related three-phase training strategy through a comparison with variants concerning the modeling and algorithmic details. In particular, we assess the roles of second-order optimization, inner-weight training, network adaptivity, and anisotropic kernel parameterizations. Numerical experiments on high-order, fractional, and anisotropic PDE benchmarks illustrate the empirical insensitivity to kernel choice, as well as the resulting trade-offs between accuracy, sparsity, and computational cost. Collectively, these results consolidate and generalize the theoretical and computational framework of SparseRBFnet, supporting accurate sparse representations with efficient operator evaluation and offering theory-grounded guidance for algorithmic and modeling choices.
\end{abstract}

\maketitle
\section{Introduction}

Recent progress in machine learning has had a growing impact on computational mathematics, motivating the development of data-driven, learning-based, and meshfree approaches as alternatives to classical mesh-based methods for solving partial differential equations (PDEs)~\cite{ainsworth2021galerkin,ainsworth2025extended,chen2021solving,e2018deep,kharazmi2021hp,lagaris1998artificial,lu2021deepxde,raissi2019physics,siegel2023greedy,zang2020weak}. In particular, the success of deep learning has motivated parameterizing PDE solutions using neural networks (e.g., Physics-Informed Neural Networks (PINNs~\cite{raissi2019physics}), Deep Ritz~\cite{e2018deep}), where the solution is represented by a deep neural network and obtained through the minimization of a loss functional, typically a discretized Ritz energy or an $L^{2}$-type PDE residual. While deep neural networks have proven universality \cite{barron1993universal,cybenko1989approximation}, they give rise to large-scale, highly nonconvex, and often ill-conditioned optimization problems. This makes hyperparameter tuning and training challenging and potentially unstable, along with additional issues such as frequency (spectral) bias \cite{cao2025analysis,krishnapriyan2021characterizing,wang2023expert,wang2021understanding,wang2022and,rahaman2019spectral,rathore2024challenges,saarinen1993ill,xu2019frequency}.  To address these difficulties, several recent approaches revisit shallow network architectures and and adopt incremental basis construction strategies \cite{bach:2017,bengio2005convex,shao2025solving,siegel2023greedy}. These methods, including boosting-style procedures, Frank--Wolfe variants, and greedy-type algorithms, build the network adaptively by repeatedly solving a nonconvex subproblem that selects a new unit to be added to the model.
This strategy substantially reduces the nonconvexity of training (provided that the subproblem can be solved effectively), removes the need to predefine the network architecture, and leads to significantly more stable optimization in practice, while retaining strong approximation capability. In addition, one can characterize the functions represented by shallow neural networks using a precise Banach space framework \cite{bartolucci2023understanding,siegel2023characterization}. 
On the other hand, kernel methods are also widely used for solving PDEs, owing to their ease of implementation and the well-developed theory of reproducing kernel Hilbert spaces (RKHSs), also known as native spaces, for the associated function spaces \cite{schaback2006kernel,wendland2004scattered}. While kernel methods were originally developed for linear tasks such as regression and linear PDEs, recent work has extended them to solving nonlinear PDEs through iterative linearizations \cite{chen2021solving}. Nevertheless, a long-standing challenge in classical kernel methods is their strong sensitivity to the choice of the kernel and associated shape parameters, as well as severe ill-conditioning and poor scalability in high dimensions.

In our previous work, we proposed a sparse radial basis function network (SparseRBFnet) framework for solving nonlinear PDEs, motivated in part by addressing the sensitivity of classical kernel methods to the choice of RBF radii \cite{shao2025solving}. The method employs an adaptive training procedure with sparsity-promoting regularization, inserting and deleting kernels with trainable centers and radii to maintain a compact representation. The resulting solution space, corresponding to the function space induced by an infinite-width (integral) neural network, admits a reproducing kernel Banach space (RKBS) structure \cite{bartolucci2023understanding}, within which the existence of finite-width solutions can be established. From an algorithmic perspective, the proposed framework lies at the intersection of kernel methods, physics-informed learning, and greedy algorithms. Each neuron in the network corresponds to a kernel basis function with trainable parameters, linking the model directly to classical kernel methods. At the same time, the solution is parameterized as a shallow neural network and trained by minimizing a PINN-style loss functional, placing the approach within the broader class of physics-informed learning methods. The adaptive insertion of kernels follows a relaxed greedy strategy akin to orthogonal greedy algorithms. A subsequent refinement step then jointly optimizes the inner and outer weights. Together, these enable controlled growth of model complexity and improved optimization stability. 

This hybrid formulation naturally raises several questions that this work aims to address. In particular, to what extent does SparseRBFnet inherit the strengths of kernel methods, physics-informed learning, and greedy algorithms, while mitigating their respective limitations? While classical kernel methods provide clear characterizations of the RKHS associated with different kernel families, it remains unclear how the corresponding RKBS induced by different classes of kernels (e.g. Gaussian and Mat\'ern kernels) should be characterized, and how these differences manifest in practical performance. Moreover, the kernel-based structure enables efficient and accurate evaluation of differential operators through explicit calculus. For integer-order derivatives, this approach reduces both computational cost and memory usage compared with automatic differentiation (auto-diff), and it extends naturally to nonlocal operators that are infeasible to handle using auto-diff, including the fractional Laplacian. A formal and unified treatment of these aspects, particularly for fractional operators, is not yet developed.
Finally, the proposed training procedure employs a greedy-type kernel insertion followed by second-order optimization, which differs fundamentally from the optimization paradigms commonly used in existing learning-based PDE solvers. Understanding the role, necessity, and potential advantages of this algorithmic structure is therefore essential for the continued development and refinement of the method. 

Motivated by these considerations, \textit{this work extends SparseRBFnet into a broader and more unified framework along three directions: function space theory, operator evaluation, and training algorithm. }
First, we provide a clearer characterization of the solution spaces induced by a broad class of RBFs. Second, we demonstrate that the explicit kernel structure enables efficient and stable evaluation of PDE operators, including both the integer-order differential operators and the nonlocal operators. Third, on the algorithmic side, we conduct an in-depth study of the previously proposed optimization algorithm, examining the role of its components (e.g., adaptive insertion, second-order optimization, inner weight training) by comparing them with their variants. Numerical experiments on representative benchmarks, including high-order PDEs, fractional Poisson problems, and viscous (anisotropic) Eikonal equations, illustrate the trade-offs among accuracy, sparsity, and computational cost and highlight regimes where the method provides clear advantages. Taken together, these findings contribute to a more principled understanding of the SparseRBFnet framework from both theoretical and algorithmic perspectives, and clarify its potential advantages in solving more challenging PDEs. 

Next, we review the related work, outline the structure of the manuscript, and summarize the main contributions of this work.
\subsection{Related studies}

\subsubsection{Deep learning approaches (PINNs and related approaches)}

Deep learning approaches, with physics-informed neural networks (PINNs) as a representative example \cite{raissi2019physics,lu2021deepxde,Shin_2023,ainsworth2021galerkin,bonito2025convergence}, have been widely applied to approximate the solutions of PDEs. While these works establish great success, they typically rely on fixed, large-scale, and highly over-parameterized neural networks, as is standard in much of the existing literature. It has been demonstrated that deep architectures outperform shallow neural networks in a wide range of contexts \cite{malach2019deeper,mhaskar2017and}. However, the deeply nested nonlinear structure of such networks renders the explicit evaluation of derivatives of the neural network ansatz practically infeasible, forcing a heavy reliance on automatic differentiation (auto-diff) \cite{baydin2018automatic}. While convenient to implement, auto-diff generally requires constructing large computational graphs, which can lead to substantial memory and computational overhead. These costs become especially prohibitive in high spatial dimensions or when evaluating high-order differential operators \cite{bettencourt2019taylor,karnakov2024solving,sirignano2018dgm}. Moreover, recent studies have observed that training PINNs for higher-order PDEs is significantly more challenging and unstable \cite{basir2022investigating,song2024does}. This limitation is further exacerbated in regimes where auto-diff is no longer applicable. For instance, standard PINN formulations cannot directly handle nonlocal operators, including fractional-order Laplacian, and therefore must switch to alternative discretization strategies, including quadrature-based finite difference schemes \cite{pang2019fpinns,pang2020npinns} or Monte Carlo-based methods \cite{guo2022monte}. This makes it difficult for deep neural network architectures to efficiently solve nonlocal equations. 

\subsubsection{Kernel-based and RBF approaches}
Our approach is closely related to kernel-based and radial basis function (RBF) methods, in that it also allows for adaptive selection of RBFs. RBF methods for solving PDEs have a long history, dating back at least to the work of Kansa \cite{kansa1990multiquadrics}.
 A key result in this area is that kernel functions induce a reproducing kernel Hilbert space (RKHS), or native space, providing a rigorous functional-analytic framework for approximation \cite{wendland2004scattered}. For instance, the RKHS associated with Mat\'ern kernels coincide with Sobolev spaces. This property is thus appealing in the PDE context, where the regularity of the solution is often known a priori and can be directly matched by an appropriate choice of kernel.

The connection between kernel-based approaches in machine learning and numerical methods for PDEs was insightfully highlighted in the survey of \cite{schaback2006kernel}.
More recently, kernel methods have been extensively studied and extended to solving nonlinear PDEs \cite{batlle2025error,chen2021solving,li2024parameter}. 
Kernel and RBF methods are also especially well suited to nonlocal and fractional-order equations, owing to the availability of quasi-analytical expressions for fractional differential operators acting on RBFs \cite{burkardt2021unified,hao2025fractional,zhuang2022radial} (see also \Cref{sec:CompDiffOp}). Despite these advantages, kernel methods face two persistent challenges: choosing suitable kernels and scales, and the severe ill-conditioning incurred as the number of centers grows. Recent efforts to address these issues can be found in \cite{chen2025sparse,nelsen2025bilevel}. Our method differs from these approaches. In particular, we allow the adaptive selection of both kernel centers and scales while keeping a sparse structure. Moreover, the associated function space is no longer a Hilbert space, but instead a Banach space. Notably, we show that for a broad class of kernels, this space is invariant with respect to the specific kernel choice (see \Cref{sec:functionspace}).

\subsubsection{Adaptive methods and greedy algorithms}
Unlike methods with fixed network architectures or kernel counts, our approach is inherently adaptive.
Related adaptive strategies have also been explored in the literature, most notably through residual-based adaptive training \cite{botvinick2025ab,huang2026adaptive,liu2022adaptive}.
A closely related line of work is greedy algorithms for PDEs
\cite{siegel2023greedy},
but our method differs in several key aspects. First, we incorporate an $\ell_{1}$ sparsity-promoting regularization to enforce sparse representations of the solution. Second, instead of performing an exact greedy selection step, which is typically an intractable global optimization problem, we adopt a relaxed strategy for selecting inner weights. Lastly, the sparsity structure also allows us to jointly optimize both the inner and the outer weights, rather than freezing the inner parameters at each iteration (see \Cref{sec:algo,sec:num} for the benefits of jointly optimizing the inner and outer weights). 
Despite these algorithmic differences, the two approaches share a common functional-analytic foundation. In particular, the solution spaces in both frameworks can be characterized as variation spaces (see \Cref{sec:functionspace}), which play a central role in the underlying approximation theory. 
For convergence results of greedy-type regression algorithms, we refer readers to \cite{barron2008approximation,devore1996some,hnatiuk2025lazifying,li2024entropy,siegel2022optimal} and the references therein.

\subsubsection{Random feature models}
Random feature models have been proposed in both kernel methods and neural networks with randomly chosen inner weights \cite{bach2023relationship,chen2022bridging,dong2021local,he2025can,liu2021random,RahimiRecht:08,zhang2025structured,zhang2023transnet,liao2024solving}. More recently, \cite{liu2025integral} establishes the optimal convergence rates of linearized shallow neural networks in Sobolev spaces with random features and $\operatorname{ReLU}^{k}$ activations, showing that they achieve the same asymptotic rates as nonlinear approximations (albeit over a larger function class) \cite{siegel2024sharp}. On the other hand, accurately approximating functions with random feature models may require a large number of features \cite{pieper2024nonuniform}. This naturally raises questions concerning the trade-off between accuracy and computational efficiency, as well as the necessity of adaptive algorithms. In \Cref{sec:fixedNNcomparison}, we conduct a detailed error and runtime analysis of our algorithm in comparison with random feature models under a fixed budget. 

\subsection{Organization of the manuscript}
The remainder of the paper is organized as follows. We first describe the SparseRBFnet model, loss, and adaptive optimization scheme (\Cref{sec:net}). We then characterize the induced infinite-width function spaces and establish the relevant equivalences (\Cref{sec:functionspace}). Next, we develop quasi-analytical kernel calculus for computing both differential and nonlocal operators (\Cref{sec:CompDiffOp}). We subsequently discuss algorithmic variants and implementation heuristics (\Cref{sec:algo}). Finally, we present numerical experiments illustrating the performance of the method on high-order, fractional, and anisotropic PDE benchmarks (\Cref{sec:num}), and make concluding remarks (\Cref{sec:conclusion}).

\noindent\textit{Summary of Contributions}. Our main contributions are summarized as follows. 
\begin{itemize}[left=1pt, label={\footnotesize$\bullet$}]
    \item We provide a theoretical characterization of the solution space induced by SparseRBFnet. Under mild assumptions on the kernel family, this space admits a description in terms of Besov spaces. Remarkably, this characterization is independent of the specific kernel choice, in contrast to classical RKHS theory, where the function space depends sensitively on the kernel.
    \item We develop an explicit operator calculus for PDEs. For integer-order differential operators, this calculus is analytical. For fractional-order operators, it takes a quasi-analytical form. The explicit kernel structure enables efficient evaluation of both local and nonlocal operators, which are difficult to handle using the standard PINNs.
    \item We study the training algorithm in detail and analyze the role of its key components. We show that adaptive kernel insertion is essential for achieving high accuracy and sparsity, and consistently outperforms fixed-width networks and random feature models. We also demonstrate that second-order optimization is critical for avoiding severe weight condensation, where kernel centers cluster and degrade approximation quality and numerical stability.
    \item We present extensive numerical experiments on a broad range of benchmark problems. These include high-order PDEs, fractional Poisson equations, and viscous (anisotropic) Eikonal equations.  Across the tested kernel families, Gaussian kernels consistently offer the best performance in terms of accuracy and robustness. 
    We also demonstrate that the performance of our method depends only mildly on the order of the PDE.
    The results illustrate the trade-offs among accuracy, sparsity, and computational cost, and demonstrate the broad applicability of the proposed framework.
\end{itemize}

\section{Network Architecture and Optimization Schemes}
\label{sec:net}
In this section, we briefly recap the problem setup and numerical approach developed in \cite{shao2025solving}, and provide necessary generalizations of the setup.

\subsection{Problem setup}

The PDEs considered in this paper are defined in
a bounded open set $D\subset \R^{d}$
and subject to “boundary conditions” on $\Upsilon\subseteq D^c$ of the following form: 
\begin{equation}
    \begin{cases}
\cE u(x) = 0,\quad & x \in D,\\
\cB u(x) = 0, \quad  &x \in \Upsilon.
\end{cases}
\end{equation}
Specifically, $\Upsilon = \partial D$ for standard PDEs and $\Upsilon = D^{c}$ for fractional-order problems such as the fractional Poisson equation. We further assume $\cE$ and $\cB$ are dependent on a series of linear operators given as
\begin{equation}
\begin{aligned}
    \cE[u](x) = \hat{E}(x, L_1[u](x), \dots, L_{N}[u](x))\\
    \cB[u](x) = \hat{B}(x, L_1[u](x), \dots, L_{N}[u](x))
\end{aligned}
\end{equation}
where $\hat{E}$, $\hat{B}$ are a nonlinear functions and $[L_i]_{i=1}^{N}$ are linear operators. 
We further assume $\hat{E}$ and $\hat{B}$ are bounded and measurable in $x$, and continuously differentiable in their remaining arguments with derivatives uniformly bounded (in $x$).

\noindent \textbf{Example.} Let $\epsilon \in \R$ and $\beta$ in $(0, 2) \cup\N$, we consider $\cE$ and $\cB$ defined by
\begin{equation}
\begin{cases}
    \cE[u](x) =  \epsilon (-\Lap)^{\beta}u(x)+ f( \nabla u(x), u(x), x) \\[0.2em]
    \cB[u](x) = u(x) - g(x)
\end{cases}.
\end{equation}
In this case, we have $N=3$ and  
\begin{equation}
    (L_{1}[u], L_{2}[u], L_{3}[u]) = ((-\Lap)^{\beta}u, \nabla u, u)
\end{equation}
and accordingly $\hat{E}: D \times \R\times  \R^{d} \times \R  \rightarrow \R$ and $\hat{B}: \Upsilon \times \R\times  \R^{d} \times \R \rightarrow \R$ are given by
\begin{equation}
    \begin{cases}
        \hat{E}(x, \cK, g, u) = \epsilon \cK + f(g, u, x)\\
        \hat{B}(x, \cK, g, u) = u - g(x)
    \end{cases}.
\end{equation}
where $f: \R^{d}\times \R\times \R^{d}\rightarrow \R$ is some nonlinear function.

\subsection{Network Design}
We consider Radial Basis Function (RBF) networks defined as
\begin{equation}
\label{eq:discrete_network}
    u_{c, \omega} (x) = \sum_{n=1}^{N} c_n \varphi(x; \omega_n),
\end{equation}
where $c= \{c_n\}_{n=1}^{N}\subseteq \R$ denotes the outer weights, and  $\omega = \{\omega_{n}\}_{n=1}^{N} \subseteq \Omega$ are the inner weights, with $\Omega$ being a prescribed parameter space. $N$ is also referred to as the network width or the number of neurons. 
The feature function $\varphi$ in this paper is defined by radial basis functions (RBFs) of the form 
\begin{equation}
\label{eq:kernelscaling}
    \varphi(x;\; y, \sigma) = \sigma^{s-d}\phi\left(\frac{|x - y|}{\sigma}\right) =  \sigma^{\gamma}\phi\left(\frac{|x - y|}{\sigma}\right).  
\end{equation}
Here, $s = d + \gamma > 0$ is a key scaling parameter that determines the properties of the function space associated with the RBF networks (see \Cref{sec:functionspace}).
In this case, the inner weights $\{ \omega_{n} = (y_n, \sigma_n) \in \R^{d} \times \R_{+} \}$ represent the centers and shapes of the RBFs. 
Throughout this paper, we assume $\Omega = \R^d\times (0, \sigma_{\max} ] $ for some $\sigma_{\max} > 0$ for the purpose of analysis. 
For practical implementation, we use a large enough compact set $\Omega_{c} =  D_1 \times [\sigma_{\min}, \sigma_{\max}] \subseteq \Omega$, where $D_1\supseteq D$, to select the inner weights.

As is widely adopted in related works, we define a measure \(\nu_D\) on \(D\) and \(\nu_{\partial D}\) on its boundary and consider the squared residual loss function 
\begin{align}
    \label{eq:loss}
    L(u) &=\frac{1}{2} \norm{\cE[u] }^2_{L^2(\nu_D)} + \frac{\lambda}{2} \norm{\cB[u]}^2_{L^2(\nu_{\partial D})}
\end{align}
where \(\lambda\) is an appropriate penalty parameter. 
In practice, we use the empirical loss function $\hat{L} = \hat{L}_{K_1, K_2}$ defined as 
\begin{equation}
\label{eq:loss_discrete}
    \hat{L}(u) =  \frac{1}{2} \sum_{k=1}^{K_1} w_{1,k}(\cE[u](x_{1,k}))^2 + \frac{\lambda}{2} \sum_{k=1}^{K_2}w_{2,k}(\cB[ u](x_{2,k}) )^2, 
\end{equation}
where $\{x_{1,k}\}_{k=1}^{K_1} \subseteq D$ and $\{x_{2,k}\}_{k=1}^{K_2} \subseteq \Upsilon$. 
For stable evaluation of $\cE[u]$ and $\cB[u]$ and subsequent optimization, we require $\gamma > \beta$ in prefactor $\sigma^{\gamma}$ to compensate for the $\sigma^{-\beta}$ growth induced by applying the operators, where $\beta$ denotes the largest order of differentiation among all linear operators $L_{i}$ (see Proposition 6 of \cite{shao2025solving} for the theoretical justification; a toy example illustrating this necessity is presented in \Cref{sec:algo}).

The neural network training is based on a sparse minimization problem with an $\ell_{1}$ regularization term:
\begin{equation}
\tag{$P_{N}^{\text{emp}}$}
\label{eq:emperical_sparse_min_discrete}
    \min_{N, c, \omega }   \hat{L}(u_{c, \omega}) + \alpha \|c\|_{1}
\end{equation}
Unlike most existing approaches, we do not fix the network width $N$ a priori; instead, it is treated as part of the optimization problem itself. Since $N$ is allowed to vary, a suitable notion of infinite-width networks is required, which is introduced in \Cref{sec:functionspace}. In \cite{shao2025solving},
a functional-analytic framework associated with infinite-width networks is developed, and the existence of a finite-width minimizer for the infinite-width formulation of the empirical problem \eqref{eq:emperical_sparse_min_discrete} is established. This result provides a rigorous justification for the adaptive sparsification strategy adopted in our method. We refer interested readers to \cite{shao2025solving} for a detailed theoretical analysis.

\subsection{Three-Phase optimization strategy}

We employ a Three-phase iterative optimization algorithm that augments standard gradient-based weights optimization with explicit kernel insertion and deletion steps, thereby dynamically adjusting the network width $N$. Below we provide a brief overview of the algorithm; a complete description can be found in Section 3 of \cite{shao2025solving}.

At each iteration, the following 3 phases are executed consecutively:\\[0.5pt]

\noindent\textbf{Phase I:} New kernel nodes are inserted in a greedy fashion, guided by a dual variable that quantifies the violation of the optimality conditions of the current sparse approximation.

\noindent\textbf{Phase II:} One step of semi-smooth Gauss-Newton optimization of $c, \omega$ ($N$ fixed) with objective
    \[
         \hat{\ell}(c, \omega) + \beta\|c\|_{1},
    \]
    where 
    \[
    \hat{\ell}(c, \omega) = \hat{L}(u_{c, \omega}).
    \]
    The $\ell_{1}$ regularization promotes sparsity in the outer weights by sending some of them to zero, and semi-smooth calculus is thus required to handle its non-smoothness.

\noindent\textbf{Phase III:} Remove kernel nodes with $c_n = 0$.\\[0.5pt]

As in most PDE learning methods, the dominant computational overhead arises from the weight optimization step in Phase II, which requires gradient evaluation ($\nabla_{c, \omega} \hat{\ell}(u_{c, \omega})$). Existing approaches whose ansatz functions are prescribed as MLP typically rely on auto-differentiation provided by machine learning backends such as PyTorch, TensorFlow, or JAX. While convenient, this strategy often produces extremely large computational graphs, particularly for high-order differential operators in high physical dimensions, leading to substantial memory overhead and degraded numerical stability. Moreover, auto-differentiation is limited to integer-order differential operators and does not naturally extend to formulations involving nonlocal or fractional-order operators, such as fractional Laplacians. 

Our framework, in contrast, exploits the closed-form analytical expressions for RBF kernels and their derivatives. This allows us to evaluate differential operators in a quasi-analytic manner, avoiding large computational graphs altogether and enabling stable treatment of fractional-order operators. We elaborate on these advantages and their practical implications in \Cref{sec:CompDiffOp}.

\section{Kernels and Function spaces}
\label{sec:functionspace}

An important feature of our approach is that the network width $N$ is not fixed a priori. 
Therefore, when we consider the limiting behavior of the network as $N\to\infty$, 
In particular, a generalization of \eqref{eq:discrete_network} to an {\it integral neural network} (\cite{bach:2017,bengio2005convex,pieper2022nonconvex,rosset2007}) is given by
\[
u_{\mu}(x) = \int_\Omega \varphi(x;\omega)\de\mu(\omega), 
\]
for $\mu\in M(\Omega)$, the space of finite signed Rodon measures on $\Omega$.
Such functions form a Banach space $\cV_\varphi^{\rm{meas}}(D)$, equipped with the norm induced by the total variation of measures:
\[
\norm{u}_{\cV_\varphi^{\rm{meas}}(D)}
= \inf \left\{\norm{\mu}_{M(\Omega)} \;\big|\; \mu \in M(\Omega)\colon u_\mu= u \text{ on } D\right\}.
\]
Here $D\subseteq \R^d$ is a spatial domain. 
The notation $\cV_\varphi^{\rm{meas}}$ emphasizes the dependence of the space on the feature function $\varphi$ and the fact that it is defined via measures.
In the literature, this function space is referred to as a reproducing kernel Banach space (RKBS) (\cite{bartolucci2023understanding}) associated with the feature map $\varphi$.
 It is also closely related to Barron spaces (\cite{wojtowytsch2022representation,e2022barron}).

Another closely related concept in the literature is the variation space (\cite{bach:2017,barron1993universal,barron2008approximation,siegel2023characterization}), 
which can be viewed as the closure of finite-width networks under an appropriate norm.
In particular, assume that the dictionary $\{ \varphi(\cdot; w) \colon w \in \Omega\}$ is included in an ambient Banach space $X$. 
Define the atomic unit ball 
   \[
   \cA_\varphi = \left\{ f: f = \sum_{n=1}^N a_n \varphi(\cdot; \omega_n),\;  N\in\N, \; \sum_{n=1}^N |a_n| \leq 1 \right\}.
   \]
   One can take the closure of $\cA_\varphi$ under the norm of $X$, denoted by $\overline{\cA_{\varphi}}:= \overline{\cA_{\varphi}}^{\|\cdot\|_X}$.
    The variation space $\cV_\varphi^{\rm{atom}}(D)$ is then defined as 
    all functions that can be scaled into $\overline{\cA_{\varphi}}$. 
More precisely, the Minkowski functional of $\overline{\cA_{\varphi}}$ defines a norm on $\cV_\varphi^{\rm{atom}}(D)$, i.e., for all $f \in \cV_\varphi^{\rm{atom}}(D)$,
\[
\| u \|_{\cV_\varphi^{\rm{atom}}(D)} = \inf \left\{ t > 0: u \in t \; \overline{\cA_{\varphi}}^{\|\cdot\|_X} \right\}.
\]
The notation $\cV_\varphi^{\rm{atom}}$ emphasizes the fact that it is defined via closure of finite atomic combinations.

The main goal of this section is to establish the equivalence of the three spaces $\cV_\varphi^{\rm{meas}}(D)$, $\cV_\varphi^{\rm{atom}}(D)$, 
and a Besov space $B^{s}_{1,1}(D)$ for an appropriate choice of the ambient space $X$ and under suitable assumptions on the kernel $\phi$.
This is important, as it provides a precise characterization of the function space induced by the RBF network ansatz.
In addition, different from existing results on Reproducing Kernel Hilbert Spaces (RKHS), which depends specifically on the kernel choice, 
here our result reveals a universal structure of the function space independent of the RBF choice. 
Our numerical results in \Cref{sec:num} also confirm that different RBFs lead to similar approximations.
Now we recall the definition of Besov spaces. Let $D \subseteq \R^d$ be a bounded Lipschitz domain.
On bounded Lipschitz domains, there are several equivalent definitions of Besov spaces. Here we adopt the integral form definition based on difference operators (see \cite[Theorem 1.118]{triebel2006}).
The Besov space $B^{s}_{1,1}(D)$ for some \(s>0\) is the set of all functions $u \in L^{1}(D)$ such that
\begin{equation}
\label{eq:BesovNormIntegralForm}
\| u \|_{B^{s}_{1,1}(D)} = \| u \|_{L^{1}(D)} + \int_0^1 t^{-s-1} \sup_{|h|\leq t}\| \Delta_{h,D}^{m} u \|_{L^{1}(D)} dt < \infty,
\end{equation}
    where $m = \lfloor s \rfloor +1$ is an integer, $\Delta_{h,D}^{m}$ is the $m$-th order difference operator with step size $h$ on the domain $D$.
By \cite[Theorem 1.105]{triebel2006} (see also \cite{rychkov1999restrictions}), there exists a bounded linear extension operator from $B^{s}_{1,1}(D)$ to $B^{s}_{1,1}(\R^d)$.
  This point is crucial, as it allows one to work on the space $\R^d$, which is necessary for applying the Littlewood--Paley characterization of Besov spaces.
    
  We now state the main assumption on the radial profile of the kernel $\phi$.   With a slight abuse of notation, we also denote by $\widehat{\phi}$ the Fourier transform of the radial function $x \mapsto \phi(|x|)$. 

  \begin{assumption}[Assumption of the kernels]
    \label{assu:kernels}
   We assume that $\phi$ satisfies the following conditions:
   \begin{enumerate}
    \item $\phi$ is continuous and the map $x \mapsto \phi(|x|)$ is in $L^p(\R^d)\cap B^s_{1,1}(\R^d)$, where $p\in [1,\infty]$.
    \item The Fourier transform $\widehat{\phi}$ of the map $x\mapsto \phi(|x|)$ satisfies $\widehat{\phi} \in W^{k,1}(\R^d)$ for some $k>d$, and that $\widehat{\phi}(\xi) \geq c > 0$ for $\xi\in B_r(0)$ for some $r>0$.
   \end{enumerate}
    \end{assumption}
\begin{remark}
    Note that the ranges \(p \geq 1\) and \(s > 0\) given above are sufficient to show the equivalence of norms considered in this section. For the general approach, stronger regularity is needed. Since we want to take derivatives in the collocation points, we need to require \(s > d + \beta\), where \(\beta \geq 0\) is the maximum differentiation order, as motivated above. This implies that the functions in \(B^s_{1,1}(\R^d)\) including their \(\beta\)-order derivatives are continuous functions. Since this regularity needs to be imposed on the underlying kernel functions \(x \mapsto \phi(|x|)\), this implies additional restriction on the choice of kernels in practice.
\end{remark}

We note that the index $p$ in Assumption \ref{assu:kernels}(1) determines the choice of the ambient space $X$, and it is always possible to choose \(p=1\) due to the definition of the Besov norm.
We will see in the following that $X = L^p(D)$ is a natural choice for the ambient space, which leads to 
a continuous map $\Phi: \Omega \to X$ defined by $\Phi(\omega) = \varphi(\cdot; \omega)$ under an additional assumption on the relation between $s$ and $p$.
In the following, 
we give some examples of kernels that satisfy Assumption \ref{assu:kernels}.\\


    \noindent \textbf{Example 1} ({\it Gaussian})  Let 
    \begin{equation}
        \phi(\rho) = \exp\left(-\frac{\rho^{2}}{2}\right).
    \end{equation}
    It is easy to check that the Gaussian kernel satisfies Assumption \ref{assu:kernels}(1) for all $s > 0$ and $1 \leq p \leq \infty$, and Assumption \ref{assu:kernels}(2) for all $k > d$ and $r>0$.\\  

    \noindent \textbf{Example 2} ({\it {\it Mat\'ern}}) Let \begin{equation}
    \phi(\rho) = \frac{2^{1-\nu}}{\Gamma(\nu)}\left(\sqrt{2\nu}\rho\right)^{\nu} K_{\nu}\left(\sqrt{2\nu}\rho\right),
    \end{equation}
    with $\nu > d/2$, where $K_{\nu}$ is the modified Bessel function of the second kind.  
 For Assumption  \ref{assu:kernels}(1), notice that standard asymptotics of $K_\nu$ give $\phi(0)=1$ and decays like $\rho^{\nu-1/2}e^{-\sqrt{2\nu}\,\rho}$ at infinity, so $x\mapsto\phi(|x|)\in L^p(\mathbb{R}^d)$ for every $1\le p\le \infty$. Moreover, the Fourier transform is given by 
 \[
\widehat{\phi}(\xi)=C_{\nu,d}\,\bigl(2\nu+|\xi|^2\bigr)^{-(\nu+d/2)}.
\]
Using Littlewood--Paley characterization of Besov spaces, we have $x\mapsto\phi(|x|)\in B^s_{1,1}(\R^d)$ for any $0<s<2\nu$. 
Additionally, from $\widehat{\phi}$, for $\alpha$ being a multi-index, we have
\[
|\partial^\alpha_\xi \widehat{\phi}(\xi)| \lesssim (1+|\xi|)^{-2\nu-d-|\alpha|}. 
\]
So $\widehat{\phi}\in W^{k,1}(\mathbb{R}^d)$ for any integer $k> d$.
Also, it is obvious that there exist $r>0$ and $c>0$ such that $\widehat{\phi}(\xi)\geq c$ for all $\xi\in B_r(0)$. 
Finally, we note that $p\in[1,\infty]$, one can pick $s$ satisfying $d(1-1/p)<s<2\nu$ when $\nu>d/2$. \\

\noindent \textbf{Example 3} ({\it Inverse multiquadrics})  
Let 
\begin{equation}
\phi(\rho)=(1+\rho^2)^{-\beta}
\end{equation}
with $\beta>d/2$.
First of all, it is obvious that $x\mapsto\phi(|x|)\in L^p(\mathbb{R}^d)$ for every $1\le p\le \infty$ with $\beta>d/2$. In addition, the Fourier transform is given by 
\[
\widehat{\phi}(\xi)=C_{d,\beta}\,|\xi|^{\beta-d/2}\,K_{\beta-d/2}(|\xi|).
\]
By the exponential decay of $K_{\beta-d/2}(|\xi|)$ as $|\xi|\to \infty$, we know the map is in $B^s_{1,1}(\R^d)$ for every $s>0$. Similarly, one can check that Assumption \ref{assu:kernels}(2) holds for all $k>d$ and some $r>0$. \\

\noindent\textbf{Example 4} (\textit{Wendland}). 
Let
\begin{equation}
\phi(\rho) = (1-\rho)_{+}^{\,m}\, P_{d,\ell}(\rho),
\end{equation}
where
\[
m := \left\lfloor \frac{d}{2} \right\rfloor + \ell + 1,
\qquad
P_{d,\ell}(\rho)
=
\sum_{k=0}^{\ell}
\binom{\ell}{k}
\binom{m+\ell+k}{\ell}
\, \rho^{k}.
\]
With this choice, $\phi \in C^{2\ell}(\mathbb{R}_{+})$ and $\supp \phi \subset [0,1]$. Since $x \mapsto \phi(|x|)$ is bounded and compactly supported, it belongs to
$L^{p}(\mathbb{R}^{d})$ for all $1 \le p \le \infty$.
Moreover, $\phi(|\cdot|) \in W^{2\ell,1}(\mathbb{R}^{d})$, and hence, by the
Sobolev--Besov embedding,
$
\phi(|\cdot|) \in B^{s}_{1,1}(\mathbb{R}^{d})
$ 
for any 
$0 < s < 2\ell$
and  Assumption~\ref{assu:kernels}(1) is satisfied.
Furthermore,
$
\widehat{\phi}(0)
=
\int_{\mathbb{R}^{d}} \phi(|x|)\,dx
>
0
$.
By continuity of $\widehat{\phi}$, there exist constants $r>0$ and $c>0$ such that
$
\widehat{\phi}(\xi) \ge c
\text{ for all } \xi \in B_r(0)
$.
Finally, for any fixed integer $k>d$, choosing $\ell$ sufficiently large (for example,
such that $2\ell > d + k$) ensures that
$
\widehat{\phi} \in W^{k,1}(\mathbb{R}^{d}),
$
and hence Assumption~\ref{assu:kernels}(2) holds.\\

    The main result of this section is the following theorem of the equivalence of the three spaces. For ease of reading, the proof of \Cref{thm:equivalence} is deferred to Appendix \ref{app:thmproof}. We also note that \cite{siegel2023characterization} derives integral formulations for certain variation spaces under broad abstract assumptions.
    Our result offers a more precise characterization of $\cV_\varphi^{\rm{atom}}(D)$ in the specific ambient space $X=L^p(D)$.
    We note that Besov space characterizations of spaces generated by superposition of Gaussian bumps have
been given in the literature. For instance, \cite[Section~6.7]{meyer1992wavelets} provides a characterization of the ``hump algebra'', which is related to $\cV_\varphi^{\rm{atom}}(\R^d)$ with Gaussian kernel and \(s = d\). The elements of the algebra are characterized in terms of their wavelet coefficients, which directly relate to the wavelet coefficient characterization of Besov spaces \cite[Section~6.10]{meyer1992wavelets}. To the best of the
authors’ knowledge, a precise equivalence theorem under conditions comparable to those considered here has not been established in the literature.

     \begin{theorem}
        \label{thm:equivalence}
        Let $D$ be a bounded Lipschitz domain in $\R^d$ and $X = L^p(D)$ for some $1 \leq p \leq \infty$. 
       Assume Assumption \ref{assu:kernels} holds, and $$s > d\left(1 - \frac{1}{p}\right).$$
Then, the three spaces $\cV_\varphi^{\rm{meas}}(D)$, $\cV_\varphi^{\rm{atom}}(D)$, and $B^{s}_{1,1}(D)$ are equivalent. 
  \end{theorem}

\section{Computation of Differential and Nonlocal Operators}
\label{sec:CompDiffOp}

In this section, we explain how the closed-form structure of the RBF ansatz allows for quasi-analytical evaluation of the (pseudo)differential operator $\cE(u_{c,\omega})$ 
and how this framework naturally extends to more challenging operators. For clarity of exposition, we focus on $\cE$ and the same procedure applies verbatim to $\cB$.

We recall $\cE$ has the form
\[
    \cE[u](x) = \hat{E}(x, L_1[u](x), \dots, L_{N}[u](x))
\]
where $\hat{E}$ is a nonlinear function and $[L_i]_{i=1}^{N}$ are linear differential operators. Accordingly, by the chain rule, the gradients $\nabla_{c, \omega} \hat{\ell}(c, \omega)$ reduce to $\nabla_{c, \omega} r^{i}(c, \omega)$ where pointwise residue $r^{i}$ is defined by 
\[
    r^{i}(c, w; x) = L_{i}[u_{c, \omega}](x).
\]
Hence, evaluating $\cE[u_{c, \omega}]$ and $\nabla_{c, \omega}\hat{\ell}(c, \omega)$ reduces to pointwise evaluation of $L_{i}[u_{c, \omega}]$ and $\nabla_{c, \omega } L_{i}[u_{c, \omega}]$ (via chain rule). Recall that our ansatz takes the form of 
\[
    u_{c, \omega} = \sum_{n=1}^{N} c_n \sigma_{n}^{\gamma} \phi(\rho(x; \omega_{n})), \qquad \rho(x; \omega_{n}) = \frac{|x - y_{n}|}{\sigma_{n}}, \qquad \omega_{n} = (y_n, \sigma_{n}).
\]
By linearity, it suffices to evaluate for a single kernel,  i.e.\ $L_{i}[\phi(\rho(\cdot; \omega_{n}))](x)$. When $L_{i}$ only involves integer-order derivatives, the closed-form of  $L_{i}[\phi(\rho(\cdot; \omega_{n}))](x)$ follows directly from the chain rule (see details in Appendix \ref{app:int-order-chainrule}). As an illustrative example, we consider the polyharmonic operators   $(-\Lap)^{m}, \ m\in\N$, where we have 
\[
(-\Lap)^{m} \!\left[\phi\!\left(\rho(\cdot;\omega_n)\right)\right](x)
=
\frac{1}{\sigma_n^{2m}}\,
\big(\cK_d^{\,m}\phi\big)\!\left(\rho(x;\omega_n)\right),
\]
where radial Laplacian operator $\cK_{d}$ is defined by
\[
\cK_d[\phi](\rho)
:=
-\phi''(\rho)
-
\frac{d-1}{\rho}\phi'(\rho).
\]
For higher-order Laplacians in a high-dimensional spatial domain (e.g. $d = 4$, $m=2$), this expression reduces the evaluation of $(-\Lap)^{m}$ to a one-dimensional differential operator acting on $\phi$, thus avoiding the need for automatic differentiation. Although analogous analytical expressions also exist for MLP ansatzes, their deeply nested and highly nonlinear structure renders such derivations impractical. Consequently, MLP-based approaches typically rely on automatic differentiation, which is computationally inefficient,
memory-intensive, and numerically unstable when evaluating high-order differential operators.

While the discussion above focuses on integer-order differential operators, many PDEs of practical interest involve nonlocal operators, most notably the fractional Laplacian. Unlike integer-order operators, fractional operators do not admit a representation in terms of pointwise derivatives and are instead defined through alternative formulations (e.g., the principle value of an integral). Remarkably, the RBF ansatz remains particularly well-suited for the fractional Laplacian, owing to its explicit radial structure. The fractional order Laplacian can again be reduced to a one-dimensional operator acting on the radial profile $\phi$, as in the integer-order scenario. We elaborate on this construction and its practical implications in the following subsection.

\subsection{Quasi-analytical evaluation of Fractional Laplacians}
The (classical and) fractional Laplacian is defined via the Fourier multiplier
\begin{equation}
\label{eq:fracLap_def}
(-\Lap)^{\beta/2} u(x)
=
\mathcal{F}^{-1}\!\left[\,|\xi|^{\beta}\,\mathcal{F}[u](\xi)\,\right](x),
\qquad \beta>0 .
\end{equation}
When $\beta \in 2\N$, the operator reduces to the classical polyharmonic operator. For $\beta \in (0,2)$, it is a standard fractional Laplacian operator; see \cite{di2012hitchhikerʼs} for an introduction. For $\beta \in  (2,\infty) \backslash 2\N$, it defines a higher-order fractional Laplacian, and we refer the readers to \cite{abatangelo2021higher} for an introduction. 
The evaluation of fractional Laplacians acting on radial basis functions has been studied extensively in the literature (see, for example, \cite{burkardt2021unified,zhuang2022radial} and references therein). We summarize here only the key results relevant to our numerical scheme and refer the interested reader to the aforementioned works for detailed derivations, theoretical analysis, and implementation considerations.

For a radial profile $\phi(\rho)$ with $\rho=|x|$, the Fourier transform is also radial and admits the Hankel-type representation
\begin{equation}
\label{eq:FTRBF}
\widehat{\phi}(\tau)
=
(2\pi)^{d/2}\,
\tau^{1-d/2}
\int_{0}^{\infty}
\rho^{d/2}\phi(\rho)\,
J_{d/2-1}(\tau \rho)\,d\rho,
\qquad \tau=|\xi|,
\end{equation}
where $J_{\nu}$ denotes the Bessel function of the first kind of order $\nu$. By \eqref{eq:fracLap_def}--\eqref{eq:FTRBF} and radial inversion, there exists a one-dimensional profile $\mathcal{K}$ such that
\begin{equation}
\label{eq:fracLap_feature}
(-\Lap)^{\beta/2}\phi(x;y,\sigma)
=
\sigma^{-\beta}\,
\mathcal{K}_{d}^{\beta/2}[\phi]\!\left(\rho(x;y,\sigma)\right),
\end{equation}
where
\begin{equation}
\label{eq:fracLapRBF}
\mathcal{K}_{d}^{\beta/2}[\phi](\rho)
=
(2\pi)^{-d/2}\,
\rho^{1-d/2}
\int_{0}^{\infty}
\tau^{d/2+\beta}\,
\widehat{\phi}(\tau)\,
J_{d/2-1}(\tau \rho)\,d\tau.
\end{equation}
Differentiating \eqref{eq:fracLapRBF} with respect to $\rho$ yields
\begin{equation}
\label{eq:drfracLapRBF}
\begin{aligned}
(\mathcal{K}_{d}^{\beta/2}[\phi])'(\rho)
&=
\left(1-d/2\right)
\rho^{-d/2}
\int_{0}^{\infty}
\tau^{d/2+\alpha}\,
\widehat{\phi}(\tau)\,
J_{d/2-1}(\tau \rho)\,d\tau \\[4pt]
&\quad
+
\rho^{1-d/2}
\int_{0}^{\infty}
\tau^{d/2+\beta+1}\,
\widehat{\phi}(\tau)\,
J'_{d/2-1}(\tau \rho)\,d\tau .
\end{aligned}
\end{equation}

We note that the above formulation works for all RBFs. Although 
\Cref{eq:FTRBF,eq:fracLapRBF,eq:drfracLapRBF} do not admit a closed form for most RBFs, the resulting one-dimensional integrals can be accurately evaluated using standard numerical quadrature, justifying the term ``quasi-analytic''. Also, unlike the FDM-based or quadrature-based methods in \cite{pang2019fpinns,pang2020npinns}, the evaluation of $\cK$ is reduced to a one-dimensional integral, with mild dependence on the physical dimension. In practice, larger $d$ can require slightly finer quadrature due to slower decay on the integrand (see \cite{ye2024fast} for related studies). 
In practice, we adopt the numerical scheme developed in \cite{zhuang2022radial} to evaluate \eqref{eq:fracLapRBF}. Since prior works do not optimize the kernel parameters $\omega$ and therefore do not require $\cK'$, we either apply the same numerical scheme to \eqref{eq:drfracLapRBF} to compute $\cK'$ directly or approximate it via finite differences of $\cK$. Although the computational cost of these one-dimensional integrals is modest, it is often convenient to precompute and tabulate both $\cK$ and $\cK'$ prior to training in order to improve efficiency.

\begin{remark}
    Nonlocal operators with finite interaction ranges have also been extensively studied in the literature \cite{DDGG20,du2019nonlocal}. Their quasi-analytical evaluation can be performed in a similar manner, provided that their Fourier symbols can be computed efficiently; see, for example, \cite{du2017fast}. 
\end{remark}

\section{Algorithmic and Model Variants: a Performance Analysis}

\label{sec:algo}


In this section, we examine the computational aspects of both the network design and the associated optimization algorithms introduced above. We first compare the proposed method with several representative algorithmic and model variants, including different optimization settings, fixed-width networks, and possibly anisotropic kernel parameterizations. We then analyze the resulting computational overhead and performance trade-offs observed in practice. Building on these observations, we conclude by discussing the classes of problems for which the proposed approach is particularly effective.


\subsection{Algorithm Variants}
From a high-level viewpoint, we prescribe the ansatz as linear combinations of kernels with prefactors $\sigma^{\gamma}$. Each iteration of the training algorithm includes a $\ell^{1}$-regularized second-order Gauss-Newton step for both inner and outer weights, along with a gradient-boosting-style greedy insertion.  While this combination works well in practice, some of the components are not canonical choices and may admit alternatives. In particular, we have the following alternatives, summarized in \Cref{tab:algorithm_variants}:

\begin{itemize}[left=1pt, label={\footnotesize$\bullet$}]
\item For implementation convenience, $\sigma$ is often strictly enforced in a closed and bounded interval. Imposing \(\sigma \geq \sigma_{\min}\), allows us to alleviate the requirement $\gamma > \beta$, where \(\beta\) is the order of the differential operator on the prefactor $\sigma^{\gamma}$, as long as the underlying radial function \(x \mapsto \phi(|x|)\) is smooth enough. Since this requirement can be relaxed (as in most of the kernel methods) from a theoretical perspective, we can also employ the simple choice \(\gamma = 0\) in practice. 

\item The Gauss–Newton step can be replaced by a first-order (proximal) gradient descent update, since the underlying optimization problem is generally nonconvex.

\item Gradient boosting is not strictly necessary. As an alternative, one may insert a randomly sampled $\omega\in \Omega$ with a probability that gradually decreases as the residual becomes smaller.

\item In standard kernel methods, the inner weights are typically fixed rather than optimized. We may choose only to optimize outer weights $c$ in Phase~II with gradient boosting (Phase~I) and pruning (Phase~III) kept. Under this setting, the resulting algorithm becomes closely related to Frank-Wolfe type methods for sparse measure optimization (see, e.g. \cite{pieper2021linear}).
\end{itemize}

\begin{table}[h!]
\centering
\caption{Algorithmic Variants}
\label{tab:algorithm_variants}
\begin{tabular}{l l l}
\hline\hline
\textbf{Category} & \textbf{Ours} & \textbf{Alternative} \\
\hline
Prefactor $\sigma^{\gamma}$  & $\gamma > \text{order of operator}$ &  No prefactor ($\gamma = 0$)         \\
Optimization      & Semi-smooth Gauss-Newton & Proximal Gradient Descent  \\
Insertion Strategy   & Boosting     & Random insertion \\
Trainable Weights & $c$ and $\omega$ & $c$ only \\
\hline\hline
\end{tabular}
\end{table}
We test these alternatives on a 2-d semilinear Poisson equation, adopting a toy example in which the exact solution consists of two localized bumps (Example 1 in Section 4.1 of \cite{shao2025solving}). We have the following observations, followed by potential explanations and discussion. 
\begin{itemize}[left=1pt, label={\footnotesize$\bullet$}]
\item The prefactor $\sigma^{\gamma}$ plays a crucial role. Removing it, by setting $\gamma = 0$, leads to noticeable performance degradation.
\item When using the first-order method, the learned kernels tend to cluster. This behavior can be attributed to widely observed weight condensation effects (see \cite{xu2025overview} for related studies).
\item Gradient boosting is essential in our framework. Due to the $\ell^{1}$-regularization, outer weights $c_n$ are driven to zero once the the outer weight gradients
enter the interval $[-\alpha,\alpha]$, similarly to proximal gradient methods, and such kernels are immediately pruned to preserve sparsity. Consequently, newly inserted kernels must be initialized with outer weights of the ``correct" sign; otherwise, they are eliminated in the next pruning step. Determining this sign requires evaluating the gradient of a candidate kernel with respect to the current network, which is precisely the role of gradient boosting. As illustrated in \Cref{fig:random_insertion}, random insertion without boosting leads to degraded performance, even under favorable conditions where all inserted kernels are assigned positive signs for a nonnegative solution\footnote{This choice is admittedly optimistic, since the sign of the solution is not known a priori in practice. Nevertheless, alternative strategies (e.g., randomly assigning signs) fail to produce even a plausible solution shape.}.
\item Training only the outer weights $c$ with sparse regularization and boosting remains viable, though it results in mildly reduced accuracy compared with jointly training $(c,\omega)$. 
\end{itemize}
\begin{figure}[t]
    \centering
    
    \begin{subfigure}[t]{0.49\textwidth}
        \centering
        \includegraphics[width=\textwidth]{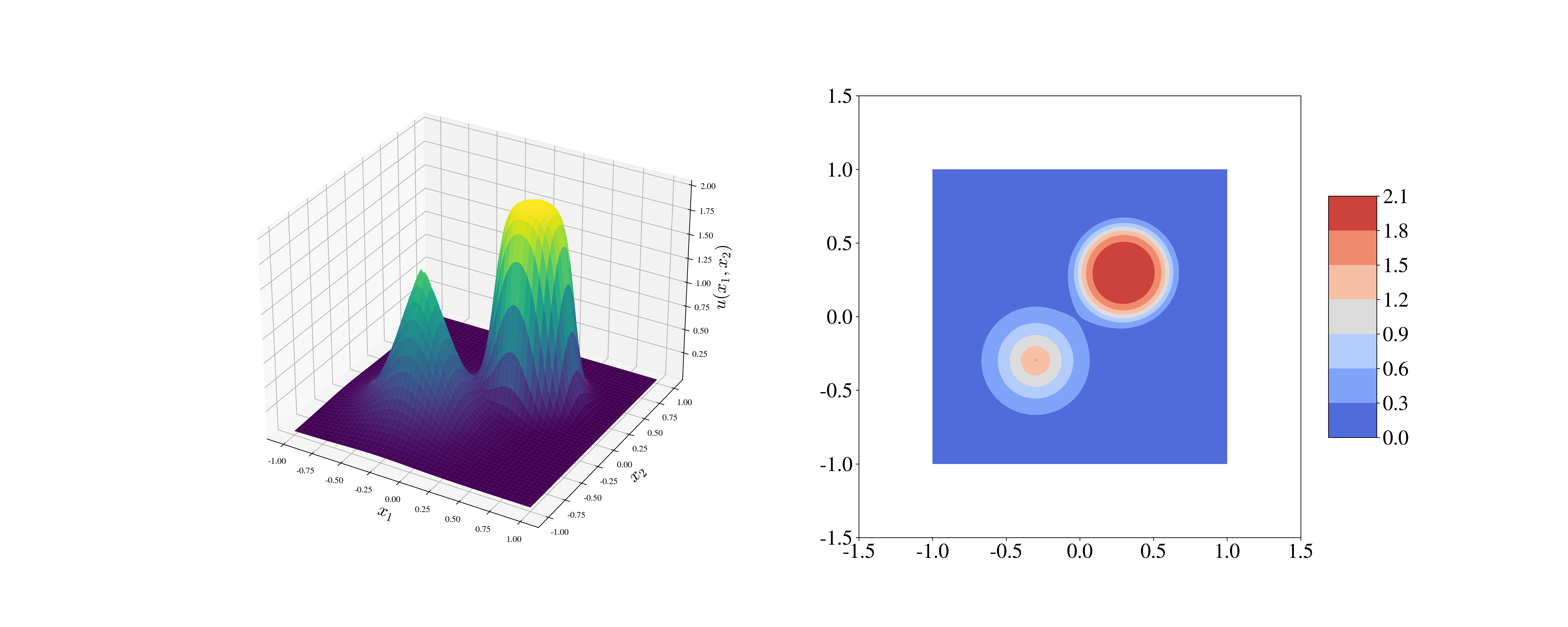}
        \caption{Exact Solution}
    \end{subfigure}
    \hfill
    \begin{subfigure}[t]{0.49\textwidth}
        \centering
        \includegraphics[width=\textwidth]{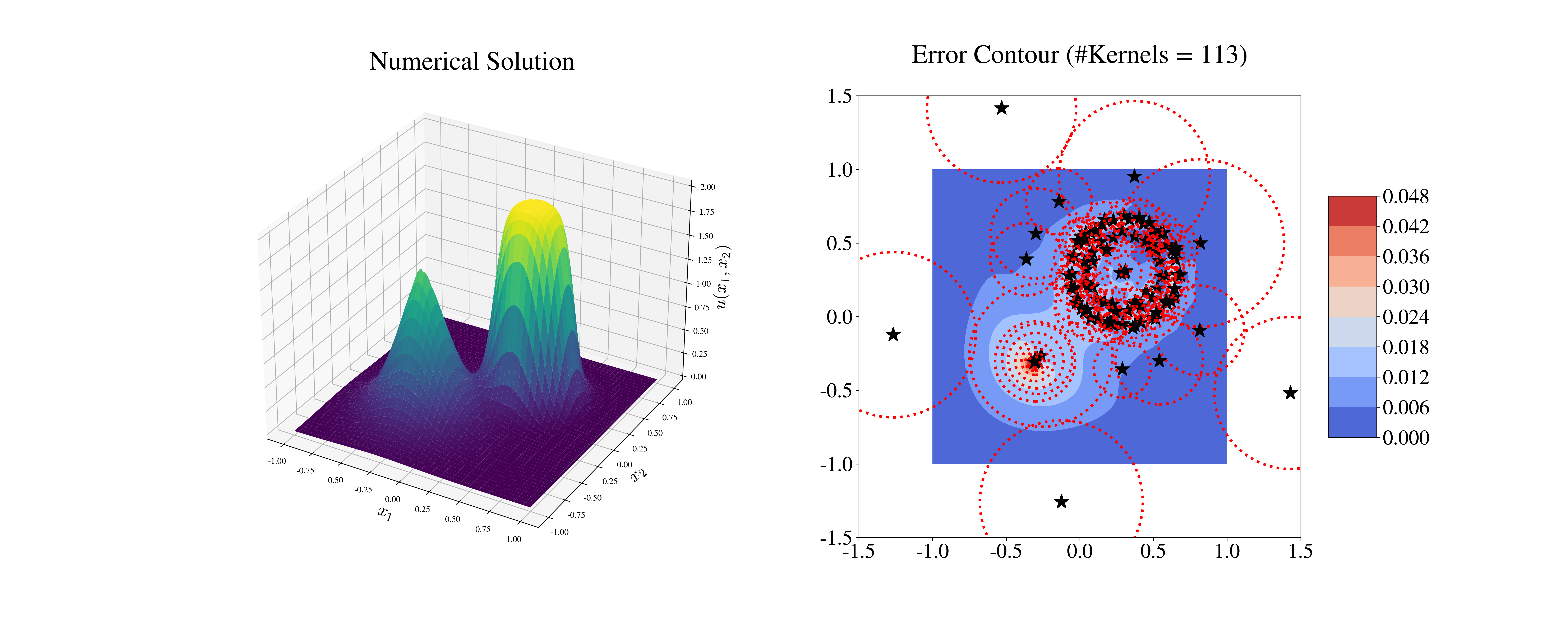}
        \caption{Second order, $c$ and $w$, gradient boosting \textbf{(ours)};\\ $\mathrm{err}_{2}^{\text{rel}} = 1.53\%$, $\#\text{params}=452$ }
    \end{subfigure}

    \begin{subfigure}[t]{0.49\textwidth}
        \centering
        \includegraphics[width=\textwidth]{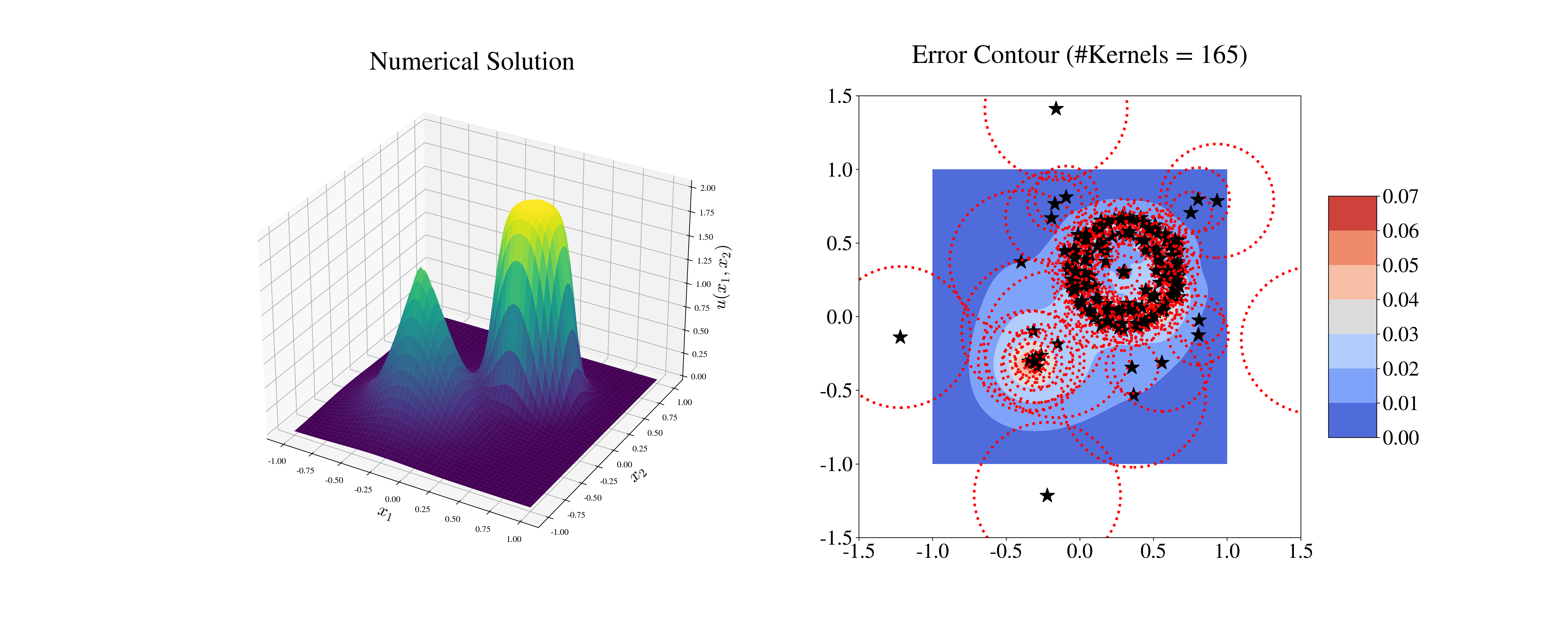}
        \caption{Second order, $c$ only, gradient boosting;\\ $\mathrm{err}_{2}^{\text{rel}} = 2.74\%$, $\#\text{params}=165$}
        \label{fig:outeronly}
    \end{subfigure}
    \hfill
    \begin{subfigure}[t]{0.49\textwidth}
        \centering
        \includegraphics[width=\textwidth]{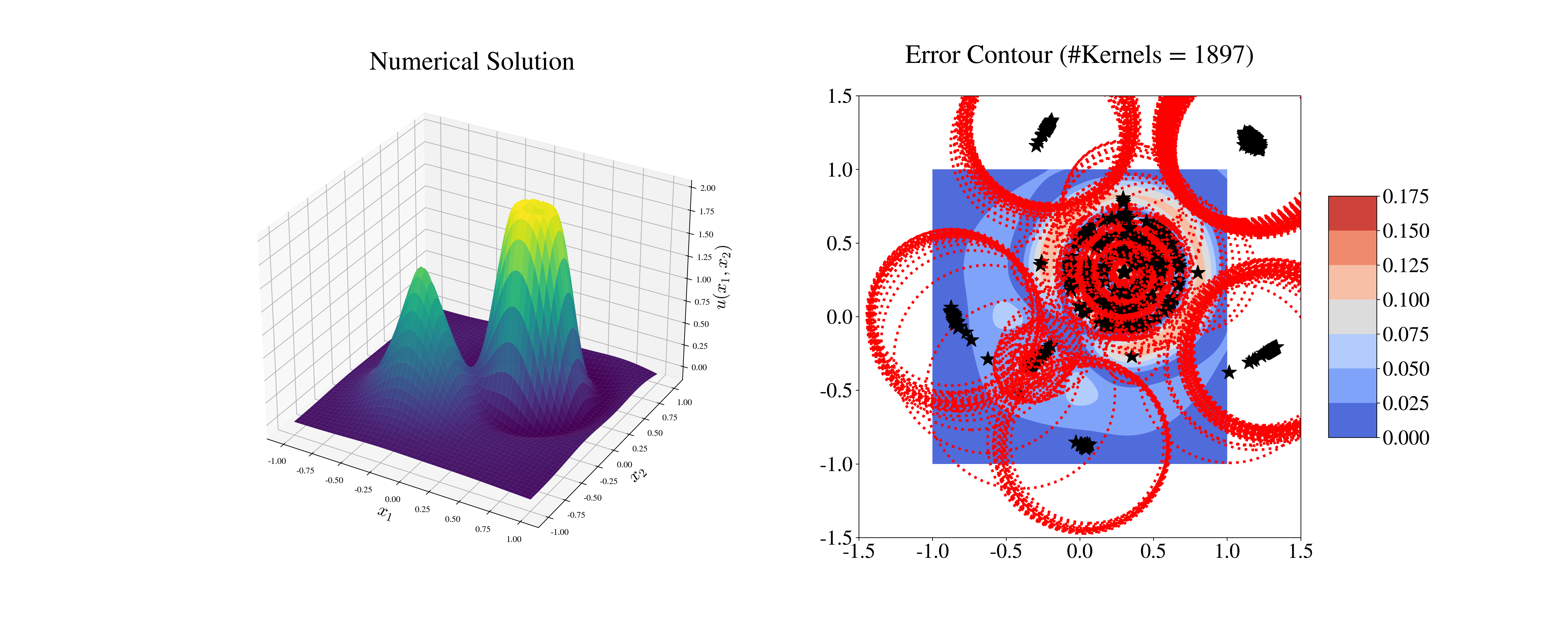}
        \caption{First order;\\ $\mathrm{err}_{2}^{\text{rel}} = 8.76\%$, $\#\text{params}=7588$}
    \end{subfigure}

    \begin{subfigure}[t]{0.49\textwidth}
        \centering
        \includegraphics[width=\textwidth]{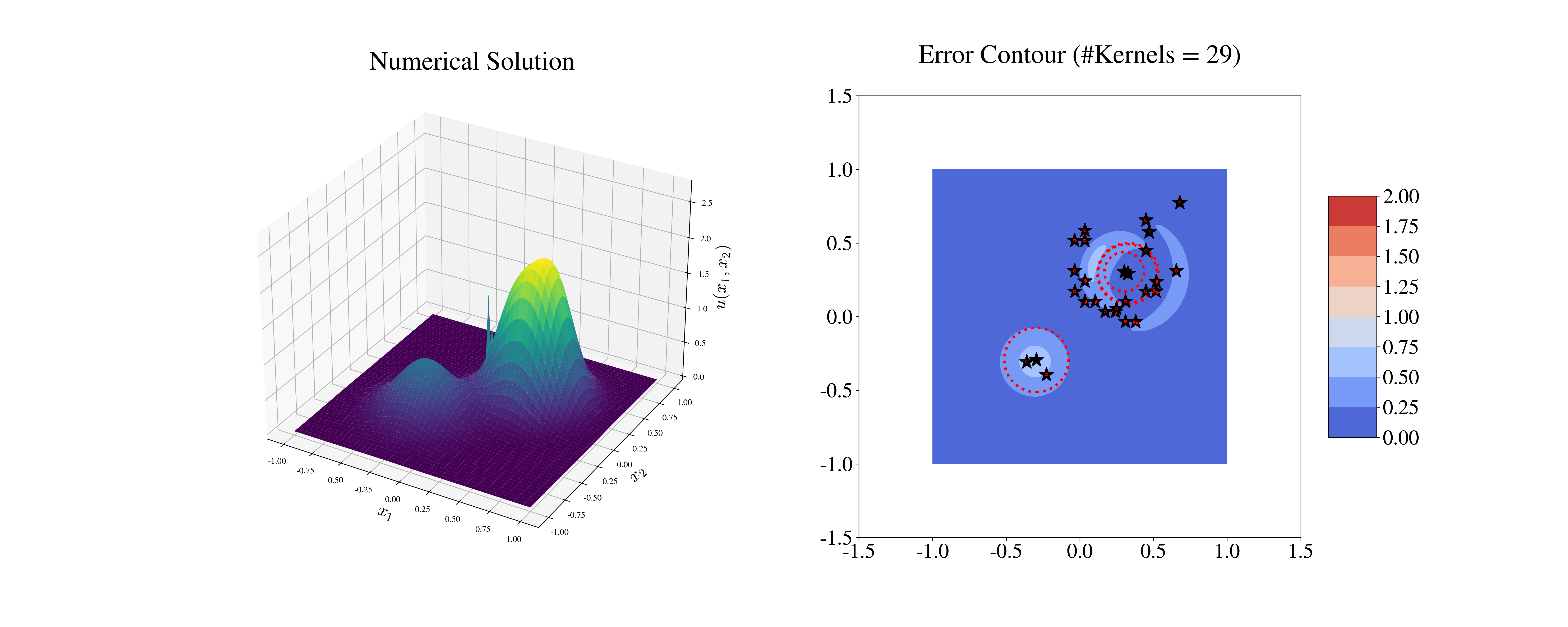}
        \caption{No prefactor $\sigma^{\gamma}$;\\ $\mathrm{err}_{2}^{\text{rel}} = 26.34\%$, $\#\text{params}=116$}
    \end{subfigure}
    \hfill
    \begin{subfigure}[t]{0.49\textwidth}
        \centering
        \includegraphics[width=\textwidth]{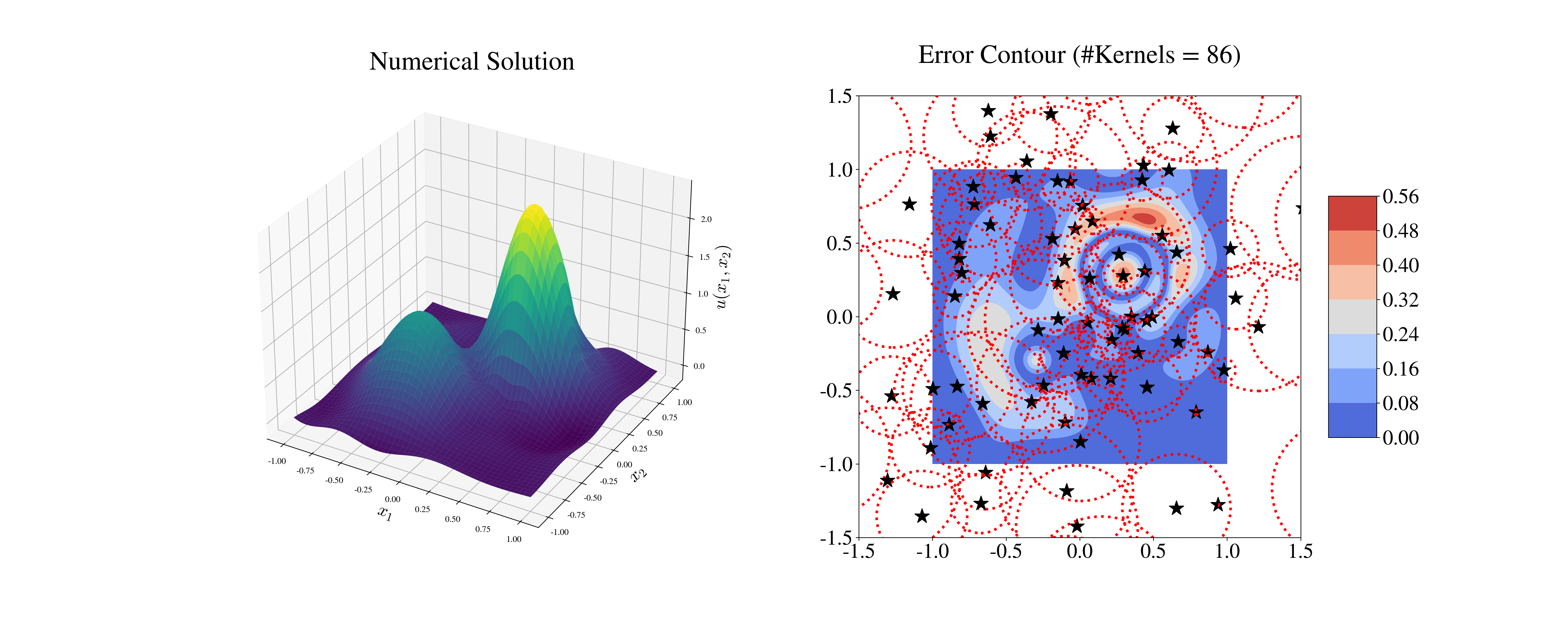}
        \caption{No gradient boosting, random insertion;\\ $\mathrm{err}_{2}^{\text{rel}} = 26.99\%$, $\#\text{params}=344$}
        \label{fig:random_insertion}
    \end{subfigure}

    \caption{Comparison between models trained to solve an elliptic PDE using different algorithmic variants. In each panel, the left plot shows the trained model on the domain, while the right plot shows the residual on the domain, together with the trained weights. Each $\star$ represents the location parameter $y_{n}$ of a learned kernel node, and the dotted circle around has radius $\sigma_{n}$.}
\label{fig:sixpanel}
\end{figure}

We extend the study in \Cref{fig:outeronly} to \Cref{sec:num}, where we examine the trade-off between accuracy and runtime. To further demonstrate the necessity of the boosting algorithm relative to random insertion without boosting, we also examine the performance of fixed-size networks with random insertion in the following subsection.

\subsection{Fixed-width network vs.\ Adaptive-width network}
\label{sec:fixedNNcomparison}
While gradient boosting is indispensable for adaptive networks, it introduces additional computational cost of order $\mathcal{O}(MK)$ per iteration, where $M$ denotes the number of candidate samples and $K$ the number of collocation points. Moreover, adaptive-width networks entail additional implementation and runtime complexity due to the need to handle dynamically changing matrix size (See Appendix~\ref{app:imp}). These considerations motivate a comparison between fixed-width networks and adaptive networks. 
We therefore train fixed-width RBF networks of varying sizes using the same number of iterations as the adaptive-width network (ours). 
The fixed networks are trained with either $\ell^{1}$ or $\ell^{2}$ regularization and are optimized using Gauss-Newton iterations. The nonsmoothness induced by the $\ell^{1}$ regularization is handled by the same semismooth scheme as employed in our adaptive method. With $\ell^{2}$ regularization, the method effectively reduces to a random feature model (see \cite[Section 1.2.2]{shao2025solving}) if the inner weights remain fixed during training.
For each run, the inner weights are sampled uniformly at random from parameter set $\Omega_{c}$, while the outer weights are initialized from a standard Gaussian distribution. Under this initialization, fixed-width networks trained with $c$-only optimization are highly sensitive to initialization and typically converge to poor solutions, analogous to the sensitivity of shape parameters in classical kernel methods, even for very large network widths.
A comparison of accuracy and computational efficiency, in which both methods employ the full solver (i.e.\ optimizing both inner and outer weights), is then summarized in \Cref{fig:fixed_vs_adaptive}.

\begin{figure}[t]
    \centering
    \begin{subfigure}[t]{0.46\textwidth}
        \centering
        \includegraphics[width=\linewidth]{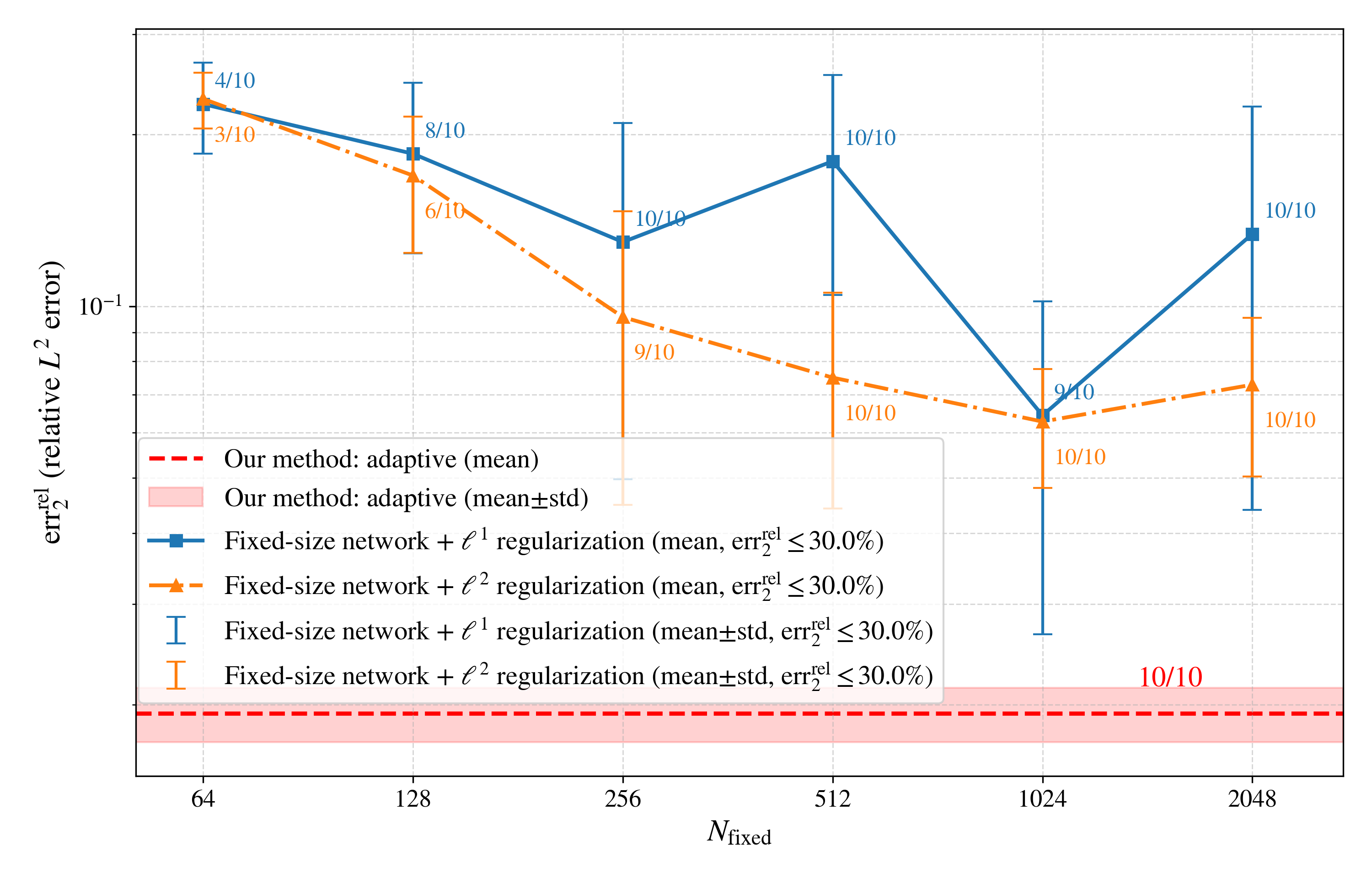}
        \caption{}
        \label{fig:fixed_error}
    \end{subfigure}
    \hfill
    \begin{subfigure}[t]{0.46\textwidth}
        \centering
        \includegraphics[width=\linewidth]{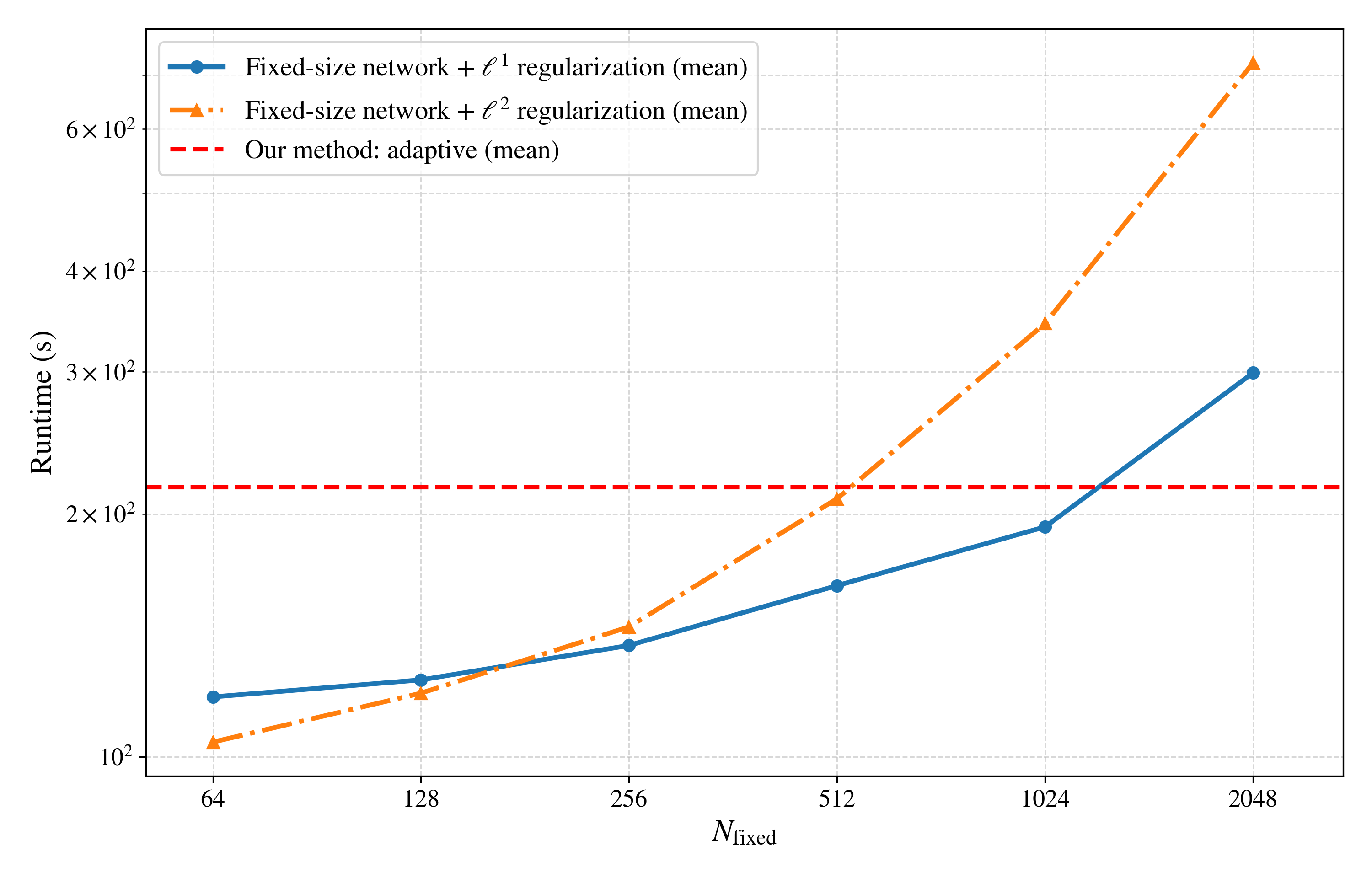}
        \caption{ }
        \label{fig:fixed_runtime}
    \end{subfigure}
    \caption{Performance comparison of fixed-size models with varying budget $N_{\mathrm{fixed}}$ (with $\ell^{1}$ and $\ell^{2}$ regularization,  respectively) 
    against the adaptive method (ours). Results aggregated from 10 runs. The regularization parameter is set to $\alpha = 10^{-3}$ for the $\ell^{1}$-regularized case and to $\alpha = 10^{-5}$ for the $\ell^{2}$-regularized case. (A) Relative test error versus fixed budget $N_{\mathrm{fixed}}$. The dashed horizontal line indicates the adaptive network benchmark (ours). Runs with relative $L^{2}$ error $(\mathrm{err}^{\mathrm{rel}}_{2})$ below $30.0\%$ are classified as successful and are retained for statistical aggregation. The corresponding success rate is indicated near each data point; (B) Runtime versus fixed budget $N_{\mathrm{fixed}}$. The dashed horizontal line indicates the adaptive network benchmark.}
    \label{fig:fixed_vs_adaptive}
\end{figure}

We observe that the adaptive network consistently achieves higher accuracy while maintaining competitive computational efficiency, despite the additional overhead introduced by gradient boosting. Moreover, the greedy insertion and deletion mechanism leads to substantially more stable training behavior. We emphasize that these conclusions are specific to the RBF network setting considered in this work, and should not be expected to generalize directly to other regularization strategies or activation functions (e.g., shallow $\operatorname{ReLU}^{k}$ networks in \cite{siegel2023greedy}).

\subsection{Anisotropic Kernels}
\label{sec:AnisoKernel}

Our isotropic kernel-based ansatz can be naturally generalized to the following anisotropic counterpart:
\begin{equation}
\label{eq:aniso_ker_def}
     \varphi_{\text{aniso}}(x;\omega) = 
    \det(\Sigma)^{\gamma/d}
    \phi\left(- |\Sigma^{-1} (x-y) | \right),
    \qquad \omega = (y,\Sigma) \in \R^{d}\times \R^{d\times d}.
\end{equation}
Here the anisotropic radius matrix $\Sigma = QR$ is further parameterized by a rotation matrix $Q$ and a diagonal matrix $R$ with entries in $(0, \sigma_{\max}]$ for some  $\sigma_{\max} >0$.
If $\Sigma = \text{diag}(\sigma, \dots, \sigma)$, $\varphi_{\text{aniso}}$ reduces to the usual isotropic case with radius $\sigma$, thereby recovering the standard isotropic definition.

For practical implementation, one may parameterize $\Sigma^{-1}$ in a brute-force way, which introduces $d(d+1)/2$ independent parameters per kernel. While increasing  number of trainable parameters associated with each kernel node, anisotropic kernels offer additional flexibility for representing directional and anisotropic features of the solution. This increased expressiveness suggests the possibility of reducing the number of kernel centers needed for a given approximation quality. In practice, however, the trade-off between expressiveness, optimization complexity, and computational efficiency is problem-dependent, and the practical benefits of anisotropy are further discussed through numerical experiments in \Cref{sec:num}.

\subsection{Computational challenges}
Several additional details have to be considered to efficiently implement the proposed methods in practice.

\subsubsection{Control of Condition number}
Classical kernel methods often suffer from ill-conditioning during optimization due to overly dense or clustered kernel centers. In our framework, this issue is largely mitigated by a sparse representation combined with a greedy insertion strategy, which inherently prevents the introduction of kernels whose parameters $\omega$ are close to those of existing ones. Instead, the dominant source of ill-conditioning arises from the disparity in gradient magnitudes between the inner and outer weights, an effect intrinsic to model design. In practice, we address this issue by employing a simple Jacobi preconditioner when solving the linear system in each Gauss–Newton step (Phase~II), resulting in consistently moderate condition numbers (typically below $10^{5}$ across all experiments).

\subsubsection{Inner weights optimization in high dimensions}

Optimizing the inner weights $\omega$ substantially increases the number of trainable variables, an effect that becomes more pronounced as the spatial dimension grows. For isotropic kernels, the total number of inner parameters scales as $O(Nd)$, whereas for anisotropic kernels it scales as $O(Nd^{2})$ ($d$ is the spatial dimension and $N$ is the number of kernels). In high-dimensional regimes, the solution geometry is typically more complex and consequently requires more kernels. Consequently, the number of inner parameters increases rapidly with the problem dimension. While the additional degrees of freedom may improve accuracy, they also render the resulting optimization problem more difficult to solve due to the increased dimensionality and nonconvexity of the parameter space.

Taken together, these considerations indicate that inner-weight optimization, despite enhanced expressivity, should be approached with care in high-dimensional problems. This observation motivates the exploration of the algorithmic variant discussed above, in which only the outer weights $c$ are trained; its performance will be examined in the numerical experiments (see \Cref{sec:num}). More broadly, the development of alternative surrogate models that avoid the brute-force parameterization of kernel centers and (anisotropic) shape parameters is of great interest and is left for future work.

\subsection{Scope of the Method}

The proposed approach is an adaptive kernel-based method. By employing a sparse representation, it is able to scale to moderately high-dimensional settings that are typically challenging for traditional kernel methods \cite{chen2025sparse}. As with other kernel-based approaches, its primary strength lies in the accurate and stable treatment of challenging differential operators, including high-order operators such as the Bi-Laplacian and nonlocal operators such as fractional Laplacians. In these cases, the ability to evaluate differential operators in a quasi-analytical manner provides a clear practical advantage over purely automatic-differentiation-based methods, such as those based on MLP ansatz.

When extending to very high-dimensional problems, however, the computational overhead associated with optimizing inner weights becomes a limiting factor. Overall, the method is best suited for moderately high-dimensional problems involving complex or nonstandard differential operators, where its meshfree formulation and quasi-analytical operator evaluation can be effectively leveraged.


\section{Numerical experiment}
\label{sec:num}

In this section, we conduct numerical experiments to demonstrate the theoretical analysis and numerical heuristics discussed above. We begin with a 4D semilinear bi-Laplacian problem with both smooth and non-smooth manufactured solutions. We then switch the focus to the fractional Poisson equation. Finally, we consider a second-order regularized Eikonal equation, which serves as a testbed for anisotropic kernels and also allows us to study the interplay between the artificial viscosity and resolution of  collocation point. Our primary accuracy metric used in these experiments are the relative $L^{2}$ error
\[
    \text{err}^{\text{rel}}_{2} = \frac{\|\hat{u} - u_{\text{true}}\|}{\|u_{\text{true}}\|},
\]
and we also record the number of kernel centers learned during training. All experiments are done on a single A100 GPU with 40GB memory, and double-precision (float64) arithmetic is used. Training time is reported as an additional indicator of computational efficiency.


\subsection{4D semilinear Bi-Laplacian Problem}

We consider the semilinear bi-Laplacian equation equipped with Navier boundary condition
i.e. 
\begin{equation}
\label{eq:semilinear-biharmonic}
\begin{cases}
(-\Lap)^{2} u(x) + u^{3}(x) = f(x) &  x\in D,\\[0.3em]
u(x) = g(x), \quad \Lap u (x) = h(x) & x\in\partial D,
\end{cases}
\end{equation}
where $D = [-1, 1]^{4}$ denotes the hyper unit cube in $\R^{4}$. Solving such an equation requires computing $(-\Lap)^{2} u$, which usually requires assembling a dummy computational graph, whereas it is analytically obtainable for our RBF networks (see Appendix \ref{app:compDiff} for comparison of computational time and memory).
Manufactured exact solutions with different levels of smoothness are tested, with $f$, $g$, and $h$ calculated accordingly.
In this example, we let $\phi$ be Gaussian and Mat\'ern kernel ($\nu = 5/2$). Although a Matérn kernel with $\nu = 9/2$ is technically required to guarantee $C^{4}$ regularity when solving fourth-order equations in strong form, the Mat\'ern kernel is smooth away from the origin. We further observe empirically that higher-order Mat\'ern kernels tend to require more delicate hyperparameter tuning.

\subsubsection{Smooth exact solution}
We first prescribe the exact solution to be 
\label{sec:bilap_smooth}
\begin{equation}
\label{eq:exact_sum_4d}
    u(x) = \sum_{i=1}^{4}\sin(\pi x_{i}), \qquad x\in D.
\end{equation}
Collocation points are chosen on uniform grids with $K_1 = 1296$ and $K_2 = 2800$ (corresponding to $8$ points per spatial dimension).
Recall that $K_1$ is the number of collocation points in the domain $D$, and $K_2$ is the number of collocation points on the boundary $\partial D$.
We consider two regularization levels, $\alpha = 1e\text{-}4$ and $\alpha = 1e\text{-}6$, and train the network using either a full solver, in which both the outer weights $c$ and inner parameters $\omega$ are optimized, or an outer-only solver mentioned in \Cref{sec:algo} (only outer weights will be trained via semi-smooth Gauss Newton in Phase~II).
For both kernel choices, the training for $\alpha = 1e\text{-}4$ is initialized from an empty network, while the $\alpha = 1e\text{-}6$ case is initialized using the solution obtained at $\alpha = 1e\text{-}4$. The results of all experiments are summarized in \Cref{tab:bilap_alpha_levels}.

We observe that the full solver consistently achieves higher accuracy with a sparser representation across both kernel classes and regularization levels. However, this improvement comes with reduced computational efficiency, owing to the increased number of trainable parameters. To further examine this trade-off, we analyze the convergence behavior for an individual training run. In \Cref{fig:loss-t-iter}, we plot the loss as a function of both iteration count and actual runtime.

When the inner parameters are optimized, the full solver exhibits a steeper decrease in the loss per iteration. Nevertheless, due to the additional cost associated with optimizing the inner weights, the outer-only solver achieves faster loss decay in terms of runtime during the early stages of training. This observation suggests that the outer-only solver can serve as an effective warm-up phase, accelerating convergence for more challenging problems before switching to the full solver for refinement. 

\begin{table}[t]
\centering
\caption{Bi-Laplacian experiment results for a smooth exact solution (mean $\pm$ std over 10 runs, 3000 iterations each).
Relative errors $\mathrm{err}^{\mathrm{rel}}_{2}$ are evaluated on a $20^{4}$ uniform grid.}
\label{tab:bilap_alpha_levels}
\renewcommand{\arraystretch}{1.15}
\begin{tabular}{|c|c|c|c|c|c|}
\hline
\textbf{Regularization} &
\textbf{Kernel} &
\textbf{Solver} &
$\mathrm{err}^{\mathrm{rel}}_{2}$ (\%) &
\textbf{\#Kernels} &
\textbf{Runtime (s)} \\
\hline

\multirow{4}{*}{$\alpha=1e\text{-}4$}
& \multirow{2}{*}{Gaussian}
& Full solver   & $1.36 \pm 0.03$ & $674 \pm 34$  & $370 \pm 2$  \\ \cline{3-6}
& 
& Outer only    & $2.63 \pm 0.14$ & $1188 \pm 19$ & $221 \pm 1$  \\ \cline{2-6}

& \multirow{2}{*}{Mat\'ern}
& Full solver   & $2.13 \pm 0.17$ & $766 \pm 63$ & $399 \pm 5$ \\ \cline{3-6}
&
& Outer only    & $4.59 \pm 0.27$ & $1874 \pm 18$ & $234 \pm 1$  \\
\hline

\multirow{4}{*}{$\alpha=1e\text{-}6$}
& \multirow{2}{*}{Gaussian}
& Full solver   & $0.48 \pm 0.06$ & $786 \pm 57$ & $524 \pm 29$ \\ \cline{3-6}
&
& Outer only    & $0.55 \pm 0.03$ & $2232 \pm 21$ & $213 \pm 1$  \\ \cline{2-6}

& \multirow{2}{*}{Mat\'ern}
& Full solver   & $0.83 \pm 0.06$ & $777 \pm 47$ & $439 \pm 124$ \\ \cline{3-6}
&
& Outer only    & $1.13 \pm 0.08$ & $2874 \pm 29$ & $266 \pm 1$  \\
\hline
\end{tabular}
\end{table}

\begin{figure}[t]
    \centering
    \begin{subfigure}[t]{0.48\textwidth}
        \centering
        \includegraphics[width=\linewidth]{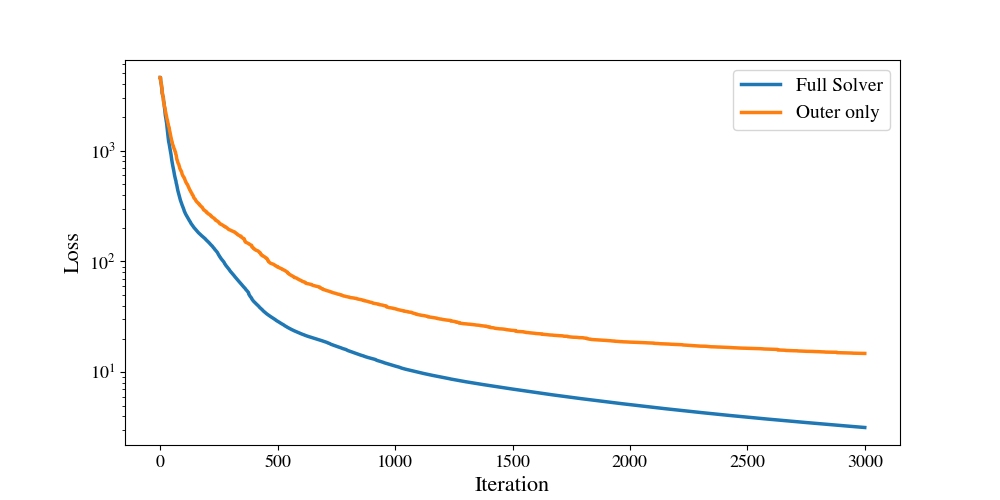}
        \caption{Loss - Iteration}
        \label{fig:loss-iter}
    \end{subfigure}
    \hfill
    \begin{subfigure}[t]{0.48\textwidth}
        \centering
        \includegraphics[width=\linewidth]{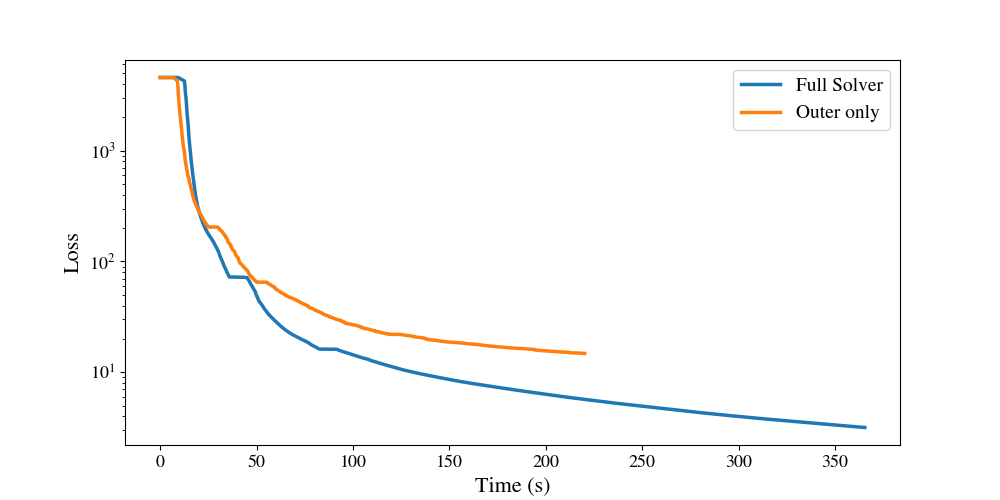}
        \caption{Loss - Time (s)}
        \label{fig:loss-t}
    \end{subfigure}
    \caption{Convergence of loss of full solver and outer only solver ($\alpha=1e\text{-}4$ with  Gaussian kernel).}
    \label{fig:loss-t-iter}
\end{figure}

\subsubsection{Low regularity exact solution}

In classical kernel methods, the approximation space is intrinsically tied to the native space of the chosen kernel. For instance, the Matérn kernel with smoothness parameter $\nu = 5/2$ is associated with the Sobolev space $H^{4.5}$, while the Gaussian kernel corresponds to a much smaller native space due to its infinite smoothness and strong decay. Consequently, traditional kernel methods are often limited in their ability to approximate solutions whose regularity falls outside these native spaces.

In contrast, as established in our analysis in \Cref{sec:functionspace}, the Banach space induced by superpositions of kernels with variable radii can be characterized as a Besov space, provided mild regularity conditions are satisfied. Notably, this resulting function space is independent of the specific kernel choice. This observation suggests that our method may accurately approximate solutions with low regularity, even with highly smooth kernels such as the Gaussian.

To further investigate this, we prescribe exact solution to be
\begin{equation}
 u(x) = \left(\frac{|x - c|}{2}\right)^{q}, \qquad x\in D;
\end{equation}
where $c = [0.5, 0.5, 0.5, 0.5]$ and the exponent $q$ explicitly controls the regularity of the solution. In particular, we consider $q \in \{4.3, \ 2.6,\ 0.3\}$, for which the corresponding solutions belong to $C^{4,0.3}$, $H^{4.5}$ and $H^{1}$, respectively. Collocation points are chosen on uniform grids with $K_1 = 10000$ and $K_2 = 2800$.

\begin{table}[t]
\centering
\caption{Bi-Laplacian experiment results for non-smooth manufactured solutions 
(mean $\pm$ std over 10 runs, 2000 iterations each).
Relative errors $\mathrm{err}^{\mathrm{rel}}_{2}$ are evaluated on a $20^{4}$ uniform grid.}
\label{tab:bilap_nonsmooth}
\renewcommand{\arraystretch}{1.15}
\begin{tabular}{|c|c|c|c|c|}
\hline
\textbf{$q$} &
\textbf{Kernel} &
$\mathrm{err}^{\mathrm{rel}}_{2}$ (\%) &
\textbf{\#Kernels} &
\textbf{Runtime (s)} \\
\hline

\multirow{2}{*}{$q = 4.3$}
& Gaussian
& $7.91 \pm 1.02$
& $1112 \pm 23$
& $459 \pm 22$ \\ \cline{2-5}

& Mat\'ern
& $5.00 \pm 0.30$
& $1632 \pm 21$
& $620 \pm 7$ \\
\hline

\multirow{2}{*}{$q = 2.6$}
& Gaussian
& $5.17 \pm 1.19$
& $1076 \pm 45$
& $422 \pm 56$ \\ \cline{2-5}

& Mat\'ern
& $3.14 \pm 0.28$
& $1629 \pm 22$
& $627 \pm 6$ \\
\hline

\multirow{2}{*}{$q = 0.3$}
& Gaussian
& $1.69 \pm 0.31$
& $1228 \pm 110$
& $468 \pm 60$ \\ \cline{2-5}

& Mat\'ern
& $1.79 \pm 0.08$
& $1558 \pm 27$
& $616 \pm 12$ \\
\hline
\end{tabular}
\end{table}

The results in \Cref{tab:bilap_nonsmooth} demonstrate that the proposed method is capable of solving Bi-Laplacian problems with low-regularity solutions. In particular, even for the case $q = 0.3$, where the manufactured solution exhibits limited smoothness, both Gaussian and Mat\'ern kernels achieve comparable accuracy, indicating that the method is not restricted by the native space regularity of the underlying kernel. For the smoother case and $q = 2.6$ and $q = 4.3$, for which the manufactured solutions belong to the native space of the Mat\'ern kernel, the Mat\'ern kernel attains lower relative error than the Gaussian kernel.
Nevertheless, this accuracy improvement is accompanied by a substantially larger number of kernel centers and, consequently, increased computational cost. When accounting for both representation complexity and runtime, the Mat\'ern kernel does not exhibit a clear overall advantage over the Gaussian kernel, even when the manufactured solution is in its native space. We also observe that the performance of both kernels degrades as $q$ increases, which is presumably due to the increased variation of the solution values along the boundary.

\subsubsection{Comparison study of solving PDEs of different orders}
We conclude this example with a comparison study for solving $0^{\text{th}}$-, $2^{\text{nd}}$-, and $4^{\text{th}}$-order PDEs. This is motivated by the known difficulty of solving PDEs of higher order for both classical methods and machine learning-based methods. For classical PDEs, the condition number of the resulting linear systems worsens as equation order becomes higher; in neural network literature, high-order differential operators are observed to corrupt back-propagated gradients and destablize convergence \cite{basir2022investigating,song2024does}. We solve the following equation 
\begin{equation}
\begin{cases}
(-\Lap)^{\beta/2} u(x)  = f(x) &  x\in D,\\[0.3em]
(-\Lap)^{l} u(x) = g_{l}(x) & x\in\partial D,\quad \forall \ l < \beta / 2 ,\ l\in\{0, 1\}
\end{cases}
\end{equation}
with $\beta \in \{0, 2, 4\}$. Exact solutions is prescribed as \Cref{eq:exact_sum_4d} and boundary conditions $g_{l}$ computed accordingly. Other computational settings are the same as in \Cref{sec:bilap_smooth} with $\alpha=1e\text{-}4$. The results are reflected in \Cref{tab:linear_poly_hamonic}. We observe that, although the relative errors are comparable across equations of different orders, our method inserts significantly more kernels when $\beta = 4$. Since the exact solution is identical in all cases, this suggests that the increased number of kernels is attributed to the intrinsic difficulty introduced by the biharmonic equation. In addition, we observe that the $0^{\text{th}}$-order (regression) problem is not necessarily easier than the second-order case and, in practice, requires explicit boundary enforcement to achieve stable and accurate solutions\footnote{We also performed the same experiment in two dimensions using a substantially larger number of collocation points to mitigate the effects of collocation sparsity in higher dimensions; the qualitative observations remain unchanged.}.  Although a systematic investigation is beyond the scope of this paper, we note that our results are consistent with recent findings on learning PDE solutions using shallow neural networks. In particular, the study in \cite{he2025can} shows that the order of the differential operator can mitigate the ill-conditioning introduced by shallow $\operatorname{ReLU}^{k}$ neural networks. This is in contrast to solving PDEs with classical FD/FE methods. This observation may help explain why learning solutions of second-order PDEs appears easier in our experiments.  For the biharmonic equation, the increased difficulty is likely attributable to the more difficult (hinged) boundary conditions.  We emphasize that this interpretation is heuristic. A more rigorous understanding of the interplay between PDE order, boundary conditions, and the conditioning behavior of RBF networks is left for future work. 

\begin{table}[t]
\centering
\caption{4D linear PDEs results under different differential operator order (mean $\pm$ std over 10 individual runs, 2000 iterations each). Relative errors $\mathrm{err}^{\mathrm{rel}}_{2}$ are evaluated on a $20^{2}$ uniform grid.}
\label{tab:linear_poly_hamonic}
\renewcommand{\arraystretch}{1.15}
\begin{tabular}{|c|c|c|c|c|}
\hline
\textbf{Kernel} &
\textbf{$\beta$} &
$\mathrm{err}^{\mathrm{rel}}_{2}$ (\%) &
\textbf{\#Kernels} &
\textbf{Runtime (s)} \\
\hline

\multirow{3}{*}{Gaussian}
& $\beta = 0$ & $1.20 \pm 0.06$ & $511.4 \pm 15.2$ & $200.8 \pm 8.2$ \\ \cline{2-5}
& $\beta = 2$ & $1.00 \pm 0.02$ & $378.4 \pm 9.9$  & $129.0 \pm 9.7$ \\ \cline{2-5}
& $\beta = 4$ & $1.51 \pm 0.05$ & $919.7 \pm 13.1$ & $262.8 \pm 0.7$ \\
\hline

\multirow{3}{*}{Mat\'ern}
& $\beta = 0$ & $1.35 \pm 0.09$ & $554.6 \pm 9.7$  & $223.0 \pm 5.5$ \\ \cline{2-5}
& $\beta = 2$ & $1.12 \pm 0.01$ & $520.3 \pm 39.8$ & $218.3 \pm 11.7$ \\ \cline{2-5}
& $\beta = 4$ & $2.18 \pm 0.22$ & $1207.1 \pm 20.5$ & $441.3 \pm 18.9$ \\
\hline

\end{tabular}
\end{table}

\subsection{Fractional Poisson equation}
In this example, we present a proof-of-concept study demonstrating how our method can be applied to fractional Laplacian problems. We consider the 2-dimensional homogeneous fractional Poisson equation with homogeneous Dirichlet exterior conditions
\begin{equation}
\begin{cases}
    (-\Lap)^{\beta/2} u(x)   = f(x) &\qquad x\in D \\[0.4em]
     u(x) \equiv 0 &\qquad x\in D^{c},
\end{cases}
\end{equation}
where $D = \{|x| <1\}$ denotes the unit disk. 
We choose the forcing term to be constant, $f(x) \equiv 1$, for which the exact solution is known and given by
\[
u(x) = 2^{-\beta}\bigl(\Gamma(1+\beta/2)\bigr)^{-2}\,(1-|x|^{2})_{+}^{\beta/2}.
\]
The training dataset is then constructed by sampling points uniformly in the radial direction on the unit disk, namely
\[
     \left\{\frac{2l}{K}(\cos(2k\pi/K), \sin(2k\pi/ K) \right\}_{0\leq  l, k\leq K},
\]
and subsequently classify the points as interior or exterior depending on whether $2l/K < 1$. Owing to the nonlocal nature of the fractional Laplacian, a relatively large number of collocation points are sampled outside the domain (see Appendix \ref{app:frac} for further details). We emphasize that this strategy is not optimal, and that more efficient approaches are available (see, e.g., \cite{burkardt2021unified, zhuang2022radial}). While such refinements are important in practice, they have been well-investigated and beyond the scope of the present manuscript.

We use Gaussian kernels and we fix $K=200$. We test $\beta = 1.5, 1.0, 0.5$. The shrinking fractional order makes the solution less regualar near the boundary ($u\in C^{0, {\beta}/{2}}$). To adapt to this, we see more kernels (with small radius) are positioned near the boundary (see \Cref{fig:frac_contour}). Also, when $\beta$ becomes smaller, the problem is essentially harder, see \Cref{tab:frac_2d_beta}. 

\begin{figure*}[t]
\centering
\begin{subfigure}{0.90\textwidth}
    \centering
    \includegraphics[width=\linewidth]{figs/frac/frac_2d_3_exact_3d.png}
    \caption{Exact solutions.}
    \label{fig:frac_exact}
\end{subfigure}

\begin{subfigure}{0.90\textwidth}
    \centering
    \includegraphics[width=\linewidth]{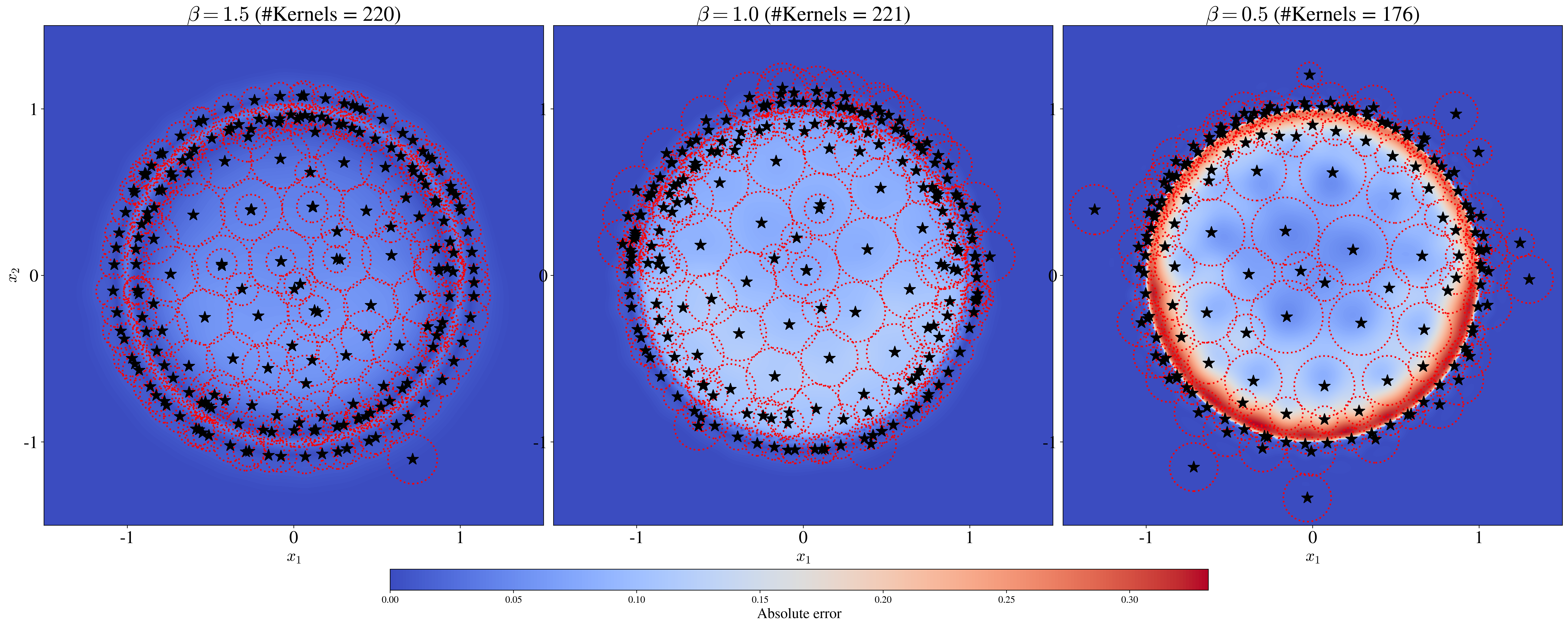}
    \caption{Error contours.}
    \label{fig:frac_contour}
\end{subfigure}
\caption{Exact solutions and error contours for the 2D fractional Poisson equation on the unit disk with different fractional orders ($\beta$). Each $\star$ in the error contour represents the
location parameter $y_{n}$ of a learned kernel node, and the dotted circle around has radius $\sigma_{n}$.}
\label{fig:frac}
\end{figure*}

\begin{table}[t]
\centering
\caption{2D fractional PDE results under different fractional order (mean $\pm$ std over 10 individual runs, 1000 iterations each).
Relative errors $\mathrm{err}^{\mathrm{rel}}_{2}$ are evaluated on a $300^{2}$ uniform test grid in $[-1, 1]^{2}$, then cropped to unit disk.}
\label{tab:frac_2d_beta}
\renewcommand{\arraystretch}{1.15}
\begin{tabular}{|c|c|c|c|}
\hline
\textbf{$\beta$ } &
$\mathrm{err}^{\mathrm{rel}}_{2}$ (\%) &
\textbf{\#Kernels} &
\textbf{Runtime (s)} \\
\hline
$\beta = 0.5$ & $11.20 \pm 0.25$ & $329 \pm 13$ & $277 \pm 2$ \\
\hline
$\beta=1.0$ & $8.70 \pm 0.32$ & $362 \pm 16$ & $280 \pm 2$ \\ 
\hline
$\beta=1.5$   & $4.05 \pm 0.04$ & $370 \pm 9$  & $268 \pm 8$ \\
\hline
\end{tabular}
\end{table}

This example highlights the capability of our method to handle nonlocal fractional Laplacian operators. Compared with the traditional kernel method, the proposed approach offers enhanced expressivity through the use of variable kernel radii, while retaining the ability to evaluate the fractional differential operator in a quasi-analytical manner. A more in-depth investigation of the resulting accuracy, efficiency, and scalability advantages will be the subject of future work.

\subsection{Eikonal Equation}

\subsubsection{2D anisotropic equation: Testbed for anisotropic kernels}
In this example, we consider the anisotropic viscous Eikonal equation
\begin{equation}
\begin{cases}
    -\epsilon \Lap u(x) + \nabla u(x)^{T} M\nabla u(x) = f(x)^{2}
    \qquad &x \in D \\
    u(x)=0 \qquad &x\in \partial D,
\end{cases}
\end{equation}
where is $M\in \R^{d\times d}$ is a symmetric positive definite matrix ($M =LL^{T} $ where $L = \begin{bmatrix}
    2.6 & 0.0 \\
    -4.0& 1.8
\end{bmatrix}$ chosen arbitrarily).
The anisotropicity induced by $M$ makes it a good testbed for anisotropic kernels. For better visualizations, we let $D = [-1, 1]^{2}$ and $f \equiv 1$. The anisotropic kernels are implemented as \eqref{eq:aniso_ker_def} in \Cref{sec:AnisoKernel}. Both isotropic kernels and anisotropic kernels are prescribed with the same training setting, and the homogeneous Dirichlet boundary condition is enforced using the ``mask function" technique as in \cite[Section~4.2]{shao2025solving}. The outcome is visualized in \Cref{fig:iso_vs_aniso} and tabulated in \Cref{tab:iso_aniso}.

\begin{figure*}[t]
    \centering
    \begin{subfigure}[t]{0.40\textwidth}
        \centering
        \includegraphics[width=\linewidth]{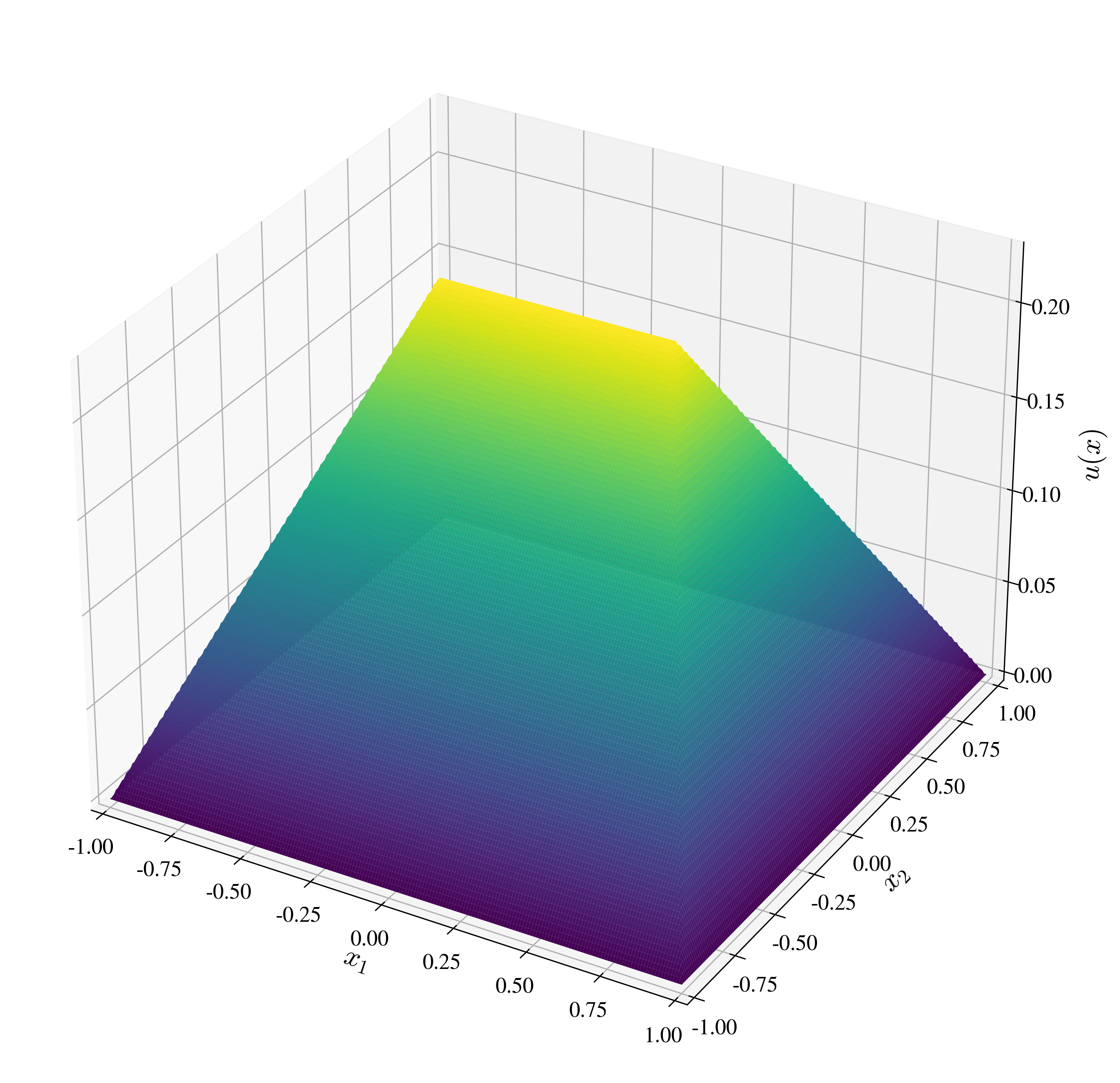}
        \caption{Viscosity solution}
        \label{fig:eikonal-true}
    \end{subfigure}
    \hfill
    \begin{subfigure}[t]{0.57\textwidth}
        \centering
        \includegraphics[width=\linewidth]{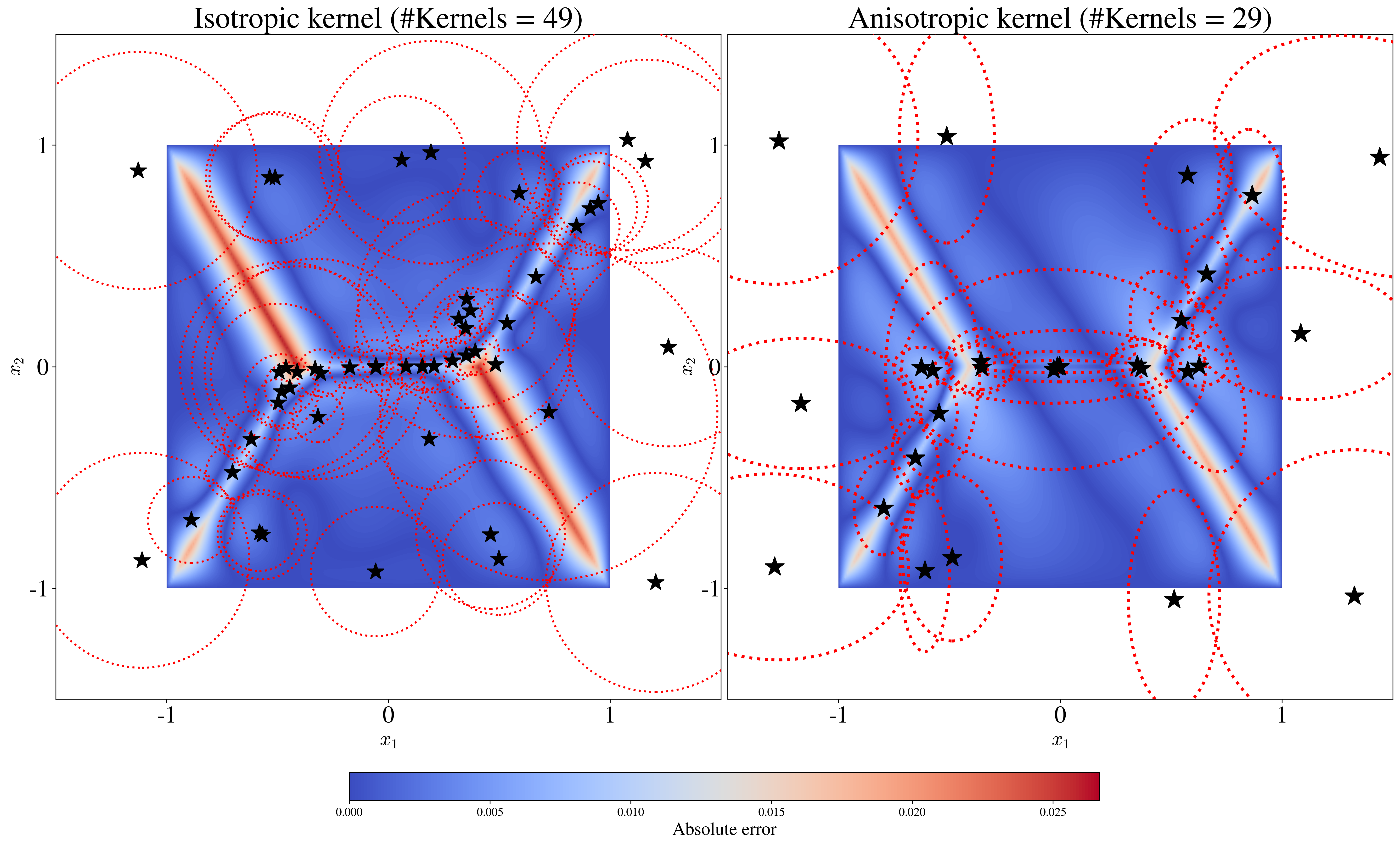}
        \caption{Error contour with isotropic \emph{(left)} and anisotropic kernels \emph{(right)} for $\alpha =1e\text{-}4$.}
        \label{fig:eikonal-comp}
    \end{subfigure}
    \caption{Comparison of isotropic and anisotropic kernels on the anisotropic viscous Eikonal equation. Each $\star$ in the error contour represents the location parameter $y_{n}$ of a learned kernel node. In the isotropic case, the dotted circle around each $\star$ has radius $\sigma_{n}$; in the anisotropic case, the dotted ellipsoid represents the level set $|\Sigma_{n}^{-1}(\cdot - y_{n})| = 1$.}
    \label{fig:iso_vs_aniso}
\end{figure*}

\begin{table}[t]
\centering
\caption{2D Eikonal experiment results with isotropic and anisotropic kernels
(mean $\pm$ std over 10 individual runs).
Relative errors $\mathrm{err}^{\mathrm{rel}}_{2}$ are evaluated on a uniform $200^2$ test grid.}
\label{tab:iso_aniso}
\renewcommand{\arraystretch}{1.15}
\begin{tabular}{|c|c|c|c|c|c|}
\hline
\textbf{Regularization} &
\textbf{Kernel} &
$\mathrm{err}^{\mathrm{rel}}_{2}$ (\%) &
\textbf{\#Kernels} &
\textbf{\#Params} &
\textbf{Runtime (s)} \\
\hline

\multirow{2}{*}{$\alpha = 1e{-}4$}
& Anisotropic
& $4.08 \pm 0.02$
& $31.0 \pm 1.60$
& $186 \pm 10$
& $122 \pm 20$ \\ \cline{2-6}

&
Isotropic
& $5.48 \pm 0.06$
& $56.3 \pm 5.55$
& $225 \pm 22$
& $131 \pm 19$ \\
\hline

\multirow{2}{*}{$\alpha = 1e{-}6$}
& Anisotropic
& $1.74 \pm 0.11$
& $86.1 \pm 5.93$
& $516 \pm 36$
& $196 \pm 2$ \\ \cline{2-6}

&
Isotropic
& $2.55 \pm 0.33$
& $89.4 \pm 9.66$
& $358 \pm 39$
& $145 \pm 3$ \\
\hline
\end{tabular}
\end{table}

We observe that, for a relatively large regularization parameter ($\alpha = 1e\text{-}4$), anisotropic kernels are able to capture directional features of the solution using substantially fewer kernels, leading to improved accuracy compared to their isotropic counterparts. However, this advantage diminishes as the regularization parameter decreases, and no clear performance gain is observed for smaller values of $\alpha$. Moreover, in both regularization regimes, although anisotropic kernels may require fewer kernel centers, they ultimately involve a comparable or even larger number of trainable parameters due to the additional matrix degrees of freedom. Consequently, the overall computational cost remains similar, or slightly higher, than that of isotropic kernels.

These results highlight the potential representational advantages of anisotropic kernels, while also underscoring the limitations of the current implementation in terms of computational efficiency. Bridging this gap will require future developments, such as low-rank parameterizations or alternative surrogate models, that can effectively reduce the cost associated with constructing and optimizing the full anisotropic matrices ${\Sigma_n^{-1}}$.

\subsubsection{1D isotropic equation: Testbed for analogy of ``resolution" in meshfree setting}

In the previous example as well as the experiments reported in our earlier work \cite{shao2025solving}, we observe that the error between the computed solution and the viscosity solution decreases with $\epsilon$ only up to a certain point. Once $\epsilon$ reaches a sufficiently small value, further reduction of $\epsilon$ no longer leads to improvement in the error, and in some cases the error even increases. This phenomenon is reminiscent of a well-known resolution constraint in classical numerical analysis for the (viscous) Eikonal equation, in which convergence to the viscosity solution is observed only when the spatial mesh size scales as $O(\epsilon)$. In our meshfree setting, a natural hypothesis is that the role of the mesh size in finite difference or finite element methods is played by the distribution of collocation points (training samples). To further investigate and validate this analogy, we conduct a controlled numerical experiment on a 1-dimensional isotropic Eikonal equation. This simple setting allows us to exhaustively refine the collocation points. The results are shown in \Cref{fig:eikonal-1d-search} and \Cref{tab:eikonal_1d_linf}. For this specific example, we use $L^{\infty}$ as the primary metric as the convergence results are usually developed in uniform norm.
\begin{figure}[t]
    \centering
    \includegraphics[width=0.8\linewidth]{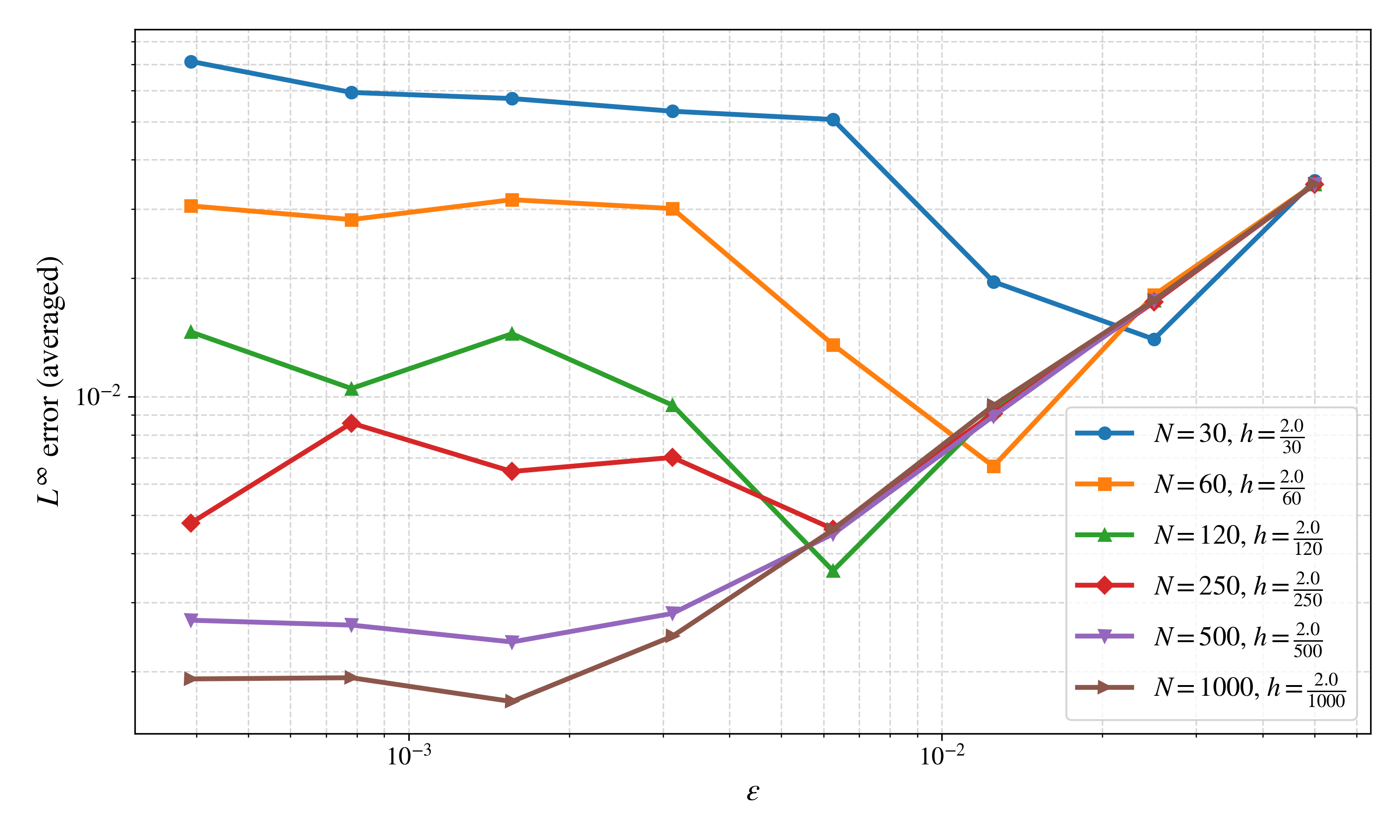}
    \caption{$L^{\infty}$ error as $\epsilon \rightarrow 0$ for different $N$ ($h$).}
    \label{fig:eikonal-1d-search}
\end{figure}

\begin{table*}[t]
\centering
\setlength{\tabcolsep}{4pt}  
\small
\caption{1D isotropic Eikonal equation.
Relative $L^\infty$ error with respect to the viscosity solution
(averaged over 10 runs, outliers excluded). 
For each $(N,h)$, the shaded entry indicates the smallest $\epsilon$ in the initial regime
where the error decays approximately linearly with $\epsilon$.}
\label{tab:eikonal_1d_linf}
\renewcommand{\arraystretch}{1.15}
\begin{tabular}{|l|c|c|c|c|c|c|c|}
\hline
$(N,h)$
& $\epsilon=\tfrac{1}{20}$
& $\epsilon=\tfrac{1}{40}$
& $\epsilon=\tfrac{1}{80}$
& $\epsilon=\tfrac{1}{160}$
& $\epsilon=\tfrac{1}{320}$
& $\epsilon=\tfrac{1}{640}$
& $\epsilon=\tfrac{1}{1280}$ \\
\hline

$(30,\tfrac{2}{30})$
& $3.5\times 10^{-2}$
& \cellcolor{gray!20}$1.4\times 10^{-2}$
& $2.0\times 10^{-2}$
& $5.1\times 10^{-2}$
& $5.3\times 10^{-2}$
& $5.7\times 10^{-2}$
& $5.9\times 10^{-2}$ \\
\hline
$(60,\tfrac{2}{60})$
& $3.5\times 10^{-2}$
& $1.8\times 10^{-2}$
& \cellcolor{gray!20}$6.7\times 10^{-3}$
& $1.4\times 10^{-2}$
& $3.0\times 10^{-2}$
& $3.2\times 10^{-2}$
& $2.8\times 10^{-2}$ \\
\hline
$(120,\tfrac{2}{120})$
& $3.5\times 10^{-2}$
& $1.8\times 10^{-2}$
& $9.3\times 10^{-3}$
& \cellcolor{gray!20}$3.6\times 10^{-3}$
& $9.5\times 10^{-3}$
& $1.4\times 10^{-2}$
& $1.0\times 10^{-2}$ \\
\hline
$(250,\tfrac{2}{250})$
& $3.5\times 10^{-2}$
& $1.7\times 10^{-2}$
& $9.1\times 10^{-3}$
& \cellcolor{gray!20}$4.6\times 10^{-3}$
& $7.0\times 10^{-3}$
& $6.5\times 10^{-3}$
& $8.6\times 10^{-3}$ \\
\hline
$(500,\tfrac{2}{500})$
& $3.5\times 10^{-2}$
& $1.8\times 10^{-2}$
& $8.9\times 10^{-3}$
& $4.5\times 10^{-3}$
& \cellcolor{gray!20}$2.8\times 10^{-3}$
& $2.4\times 10^{-3}$
& $2.6\times 10^{-3}$ \\
\hline
$(1000,\tfrac{2}{1000})$
& $3.5\times 10^{-2}$
& $1.8\times 10^{-2}$
& $9.5\times 10^{-3}$
& $4.6\times 10^{-3}$
& $2.5\times 10^{-3}$
&\cellcolor{gray!20}$1.7\times 10^{-3}$
& $1.9\times 10^{-3}$ \\

\hline
\end{tabular}
\end{table*}

We see from \Cref{fig:eikonal-1d-search}  that the logarithm of error decays linearly with $\log(\epsilon)$ until it reaches some $\epsilon$. We further mark the minimal $\epsilon$ before which the error decays approximately linearly. The pattern roughly agrees with the classical results. To further draw a quantitative conclusion, a refined line search is needed. Nevertheless, it is clear that the number, or more generally the distribution, of collocation points plays a significant role, analogous to the mesh size in FD/FE methods. To take a step further, this example indicates the importance of the distribution of collocation points. A related study can be found in \cite{esteve2025finite}.

\section{Concluding Remarks}
\label{sec:conclusion}

SparseRBFnet can be viewed as a hybrid formulation that combines ideas from kernel methods, physics-informed learning, and greedy algorithms. In this framework, the solution is represented as a sparse superposition of radial basis functions with trainable centers and shape parameters, and is obtained adaptively through a greedy insertion and pruning strategy. As a result, SparseRBFnet naturally inherits favorable properties from each of these perspectives. Building upon this, this work advanced the SparseRBFnet framework through a systematic theoretical and computational study. 
From a theoretical standpoint, we showed that the solution space induced by SparseRBFnet is equivalent to a Besov space, independent of the specific kernel choice, provided that mild assumptions hold on the radial profiles of the kernels. This characterization highlights the flexibility of the framework beyond the Hilbert space setting associated with the classical kernel methods. 
From a computational perspective, we first demonstrated that the explicit kernel structure  enables quasi-analytical evaluation of differential and nonlocal operators. This feature is particularly advantageous for problems involving high-order or nonlocal operators, where purely auto-diff schemes often encounter practical limitations. In addition, we conduct an in-depth study of the associated training algorithms, examining the roles of second-order optimization, inner-weight parameterization, and network adaptivity, etc. Closely related algorithmic and modeling alternatives, including fixed-width networks and anisotropic kernel parameterizations, are also systematically assessed. Numerical experiments are conducted on a four-dimensional bi-Laplacian problem, a fractional Laplacian equation, and viscous (anisotropic) Eikonal equations, showing the effectiveness of quasi-analytical evaluation for high-order and fractional operators, as well as the added expressiveness of anisotropic kernel parameterizations. In particular, the results indicate that different kernel choices can yield comparable performance across solutions, even when the solution regularity does not align with the native spaces of the kernels, consistent with the function-space theory developed herein. A comparison study of our method on $0^{\text{th}}$-, $2^{\text{nd}}$-, and $4^{\text{th}}$-order linear differential equations is also conducted, indicating that the performance of the method depends mildly on the PDE order. The 1D Eikonal testbed shows that convergence to the viscosity solution is governed by both the viscosity parameter and the distribution of collocation points (analogous to mesh resolution).

Further theoretical work is needed to formalize the above-mentioned interaction between the distribution of collocation points and regularization parameters. Developing a unified theory on resolution constraints in the meshfree setting would provide valuable guidance for machine learning–based PDE solvers. Also, beyond approximation-theoretic results developed in this work, a rigorous convergence analysis of the greedy training procedure, including kernel insertion and pruning under $\ell^1$-type regularization, is largely unexplored. Establishing conditions for convergence of the adaptive procedure, as well as understanding its optimization dynamics, would further strengthen the theoretical foundations of SparseRBFnet.

Further improvements in the computational efficiency of SparseRBFnet hinge on mitigating the challenges of optimizing inner weights as the problem dimension increases, in both isotropic and anisotropic kernel settings. As the dimensionality grows, jointly optimizing kernel centers and shapes leads to a substantially more challenging nonconvex optimization problem, particularly when second-order scheme is required. This motivates several promising directions for future algorithmic development. First, to mitigate the computational overhead of full second-order optimization, semi-smooth Gauss–Newton methods combined with limited-memory quasi-Newton schemes (e.g. L-BFGS) may provide an effective compromise between robustness and efficiency. Second, for anisotropic kernels in particular, surrogate or reduced-parameter models could be introduced to avoid brute-force parameterization of high-dimensional shape matrices, while still capturing essential directional features of the solution. Finally, although outer-weight-only training strategies can be viable in certain regimes, they still rely on accurate gradient estimation for kernel insertion and pruning. Developing more effective boosting or gradient-selection schemes may therefore further improve scalability and robustness.

\section*{Acknowledgements}
Z.~Shao and X.~Tian were supported in part by NSF DMS-2240180 and the Alfred P. Sloan Fellowship.
This manuscript has been co-authored by UT-Battelle, LLC, under contract DE-AC05-00OR22725 with the US Department of Energy (DOE). The US government retains and the publisher, by accepting the article for publication, acknowledges that the US government retains a nonexclusive, paid-up, irrevocable, worldwide license to publish or reproduce the published form of this manuscript, or allow others to do so, for US government purposes. DOE will provide public access to these results of federally sponsored research in accordance with the DOE Public Access Plan.

The authors thank Hongkai Zhao, Qiang Du, and Rahul Parhi for helpful discussions. 

\section*{Data availability}
The datasets generated during and/or analysed during the current study are available
from the corresponding author on reasonable request.

\bibliographystyle{acm}
\bibliography{refs}

\begin{thebibliography}{10}

\bibitem{abatangelo2021higher}
{\sc Abatangelo, N.}
\newblock Higher-order fractional {L}aplacians: an overview.
\newblock In {\em Bruno Pini Mathematical Analysis Seminar\/} (2021), vol.~12, pp.~53--80.

\bibitem{ainsworth2021galerkin}
{\sc Ainsworth, M., and Dong, J.}
\newblock Galerkin neural networks: {{A}} framework for approximating variational equations with error control.
\newblock {\em SIAM Journal on Scientific Computing 43}, 4 (2021), A2474--A2501.

\bibitem{ainsworth2025extended}
{\sc Ainsworth, M., and Dong, J.}
\newblock Extended {{Galerkin}} neural network approximation of singular variational problems with error control.
\newblock {\em SIAM Journal on Scientific Computing 47}, 3 (2025), C738--C768.

\bibitem{bach:2017}
{\sc Bach, F.}
\newblock Breaking the curse of dimensionality with convex neural networks.
\newblock {\em J. Mach. Learn. Res. 18\/} (2017), 53.
\newblock Id/No 19.

\bibitem{bach2023relationship}
{\sc Bach, F.}
\newblock On the relationship between multivariate splines and infinitely-wide neural networks, 2023.

\bibitem{barron1993universal}
{\sc Barron, A.~R.}
\newblock Universal approximation bounds for superpositions of a sigmoidal function.
\newblock {\em IEEE Transactions on Information theory 39}, 3 (1993), 930--945.

\bibitem{barron2008approximation}
{\sc Barron, A.~R., Cohen, A., Dahmen, W., and DeVore, R.~A.}
\newblock Approximation and learning by greedy algorithms.
\newblock {\em Ann. Statist. 36}, 1 (2008), 64--94.

\bibitem{bartolucci2023understanding}
{\sc Bartolucci, F., De~Vito, E., Rosasco, L., and Vigogna, S.}
\newblock Understanding neural networks with reproducing kernel {{Banach}} spaces.
\newblock {\em Applied and Computational Harmonic Analysis 62\/} (2023), 194--236.

\bibitem{basir2022investigating}
{\sc Basir, S.}
\newblock Investigating and mitigating failure modes in physics-informed neural networks ({PINNs}).
\newblock {\em arXiv preprint arXiv:2209.09988\/} (2022).

\bibitem{batlle2025error}
{\sc Batlle, P., Chen, Y., Hosseini, B., Owhadi, H., and Stuart, A.~M.}
\newblock Error analysis of kernel/{{GP}} methods for nonlinear and parametric {{PDEs}}.
\newblock {\em Journal of Computational Physics 520\/} (2025), 113488.

\bibitem{baydin2018automatic}
{\sc Baydin, A.~G., Pearlmutter, B.~A., Radul, A.~A., and Siskind, J.~M.}
\newblock Automatic differentiation in machine learning: a survey.
\newblock {\em Journal of machine learning research 18}, 153 (2018), 1--43.

\bibitem{bengio2005convex}
{\sc Bengio, Y., Roux, N., Vincent, P., Delalleau, O., and Marcotte, P.}
\newblock Convex neural networks.
\newblock {\em Advances in neural information processing systems 18\/} (2005).

\bibitem{bettencourt2019taylor}
{\sc Bettencourt, J., Johnson, M.~J., and Duvenaud, D.}
\newblock Taylor-mode automatic differentiation for higher-order derivatives in jax.
\newblock In {\em Program Transformations for ML Workshop at NeurIPS 2019\/} (2019).

\bibitem{bonito2025convergence}
{\sc Bonito, A., DeVore, R., Petrova, G., and Siegel, J.~W.}
\newblock Convergence and error control of consistent pinns for elliptic pdes.
\newblock {\em IMA Journal of Numerical Analysis\/} (2025), draf008.

\bibitem{botvinick2025ab}
{\sc {Botvinick-Greenhouse}, J., Ali, W.~H., Benosman, M., and Mowlavi, S.}
\newblock {{AB-pinns}}: {{Adaptive-basis}} physics-informed neural networks for residual-driven domain decomposition.
\newblock {\em arXiv preprint arXiv:2510.08924\/} (2025).

\bibitem{burkardt2021unified}
{\sc Burkardt, J., Wu, Y., and Zhang, Y.}
\newblock A unified meshfree pseudospectral method for solving both classical and fractional pdes.
\newblock {\em SIAM Journal on Scientific Computing 43}, 2 (2021), A1389--A1411.

\bibitem{cao2025analysis}
{\sc Cao, W., and Zhang, W.}
\newblock An analysis and solution of ill-conditioning in physics-informed neural networks.
\newblock {\em Journal of Computational Physics 520\/} (2025), 113494.

\bibitem{chen2022bridging}
{\sc Chen, J., Chi, X., Yang, Z., et~al.}
\newblock Bridging traditional and machine learning-based algorithms for solving pdes: the random feature method.
\newblock {\em J Mach Learn 1}, 3 (2022), 268--298.

\bibitem{chen2021solving}
{\sc Chen, Y., Hosseini, B., Owhadi, H., and Stuart, A.~M.}
\newblock Solving and learning nonlinear {PDEs} with gaussian processes.
\newblock {\em Journal of Computational Physics 447\/} (2021), 110668.

\bibitem{chen2025sparse}
{\sc Chen, Y., Owhadi, H., and Sch{\"a}fer, F.}
\newblock Sparse cholesky factorization for solving nonlinear {PDEs} via {G}aussian processes.
\newblock {\em Mathematics of Computation 94}, 353 (2025), 1235--1280.

\bibitem{cybenko1989approximation}
{\sc Cybenko, G.}
\newblock Approximation by superpositions of a sigmoidal function.
\newblock {\em Mathematics of control, signals and systems 2}, 4 (1989), 303--314.

\bibitem{DDGG20}
{\sc D'Elia, M., Du, Q., Glusa, C., Gunzburger, M., Tian, X., and Zhou, Z.}
\newblock Numerical methods for nonlocal and fractional models.
\newblock {\em Acta Numerica 29\/} (2020), 1--124.

\bibitem{devore1996some}
{\sc DeVore, R.~A., and Temlyakov, V.~N.}
\newblock Some remarks on greedy algorithms.
\newblock {\em Advances in computational Mathematics 5}, 1 (1996), 173--187.

\bibitem{di2012hitchhikerʼs}
{\sc Di~Nezza, E., Palatucci, G., and Valdinoci, E.}
\newblock Hitchhiker's guide to the fractional {{Sobolev}} spaces.
\newblock {\em Bulletin des sciences math\'ematiques 136}, 5 (2012), 521--573.

\bibitem{dong2021local}
{\sc Dong, S., and Li, Z.}
\newblock Local extreme learning machines and domain decomposition for solving linear and nonlinear partial differential equations.
\newblock {\em Computer Methods in Applied Mechanics and Engineering 387\/} (2021), 114129.

\bibitem{du2019nonlocal}
{\sc Du, Q.}
\newblock {\em Nonlocal Modeling, Analysis, and Computation: Nonlocal Modeling, Analysis, and Computation}.
\newblock SIAM, 2019.

\bibitem{du2017fast}
{\sc Du, Q., and Yang, J.}
\newblock Fast and accurate implementation of fourier spectral approximations of nonlocal diffusion operators and its applications.
\newblock {\em Journal of Computational Physics 332\/} (2017), 118--134.

\bibitem{e2022barron}
{\sc E, W., Ma, C., and Wu, L.}
\newblock The {Barron} space and the flow-induced function spaces for neural network models.
\newblock {\em Constructive Approximation 55}, 1 (2022), 369--406.

\bibitem{e2018deep}
{\sc E, W., and Yu, B.}
\newblock The deep {{Ritz}} method: A deep learning-based numerical algorithm for solving variational problems.
\newblock {\em Communications in Mathematics and Statistics 6}, 1 (2018), 1--12.

\bibitem{esteve2025finite}
{\sc Esteve-Yag{\"u}e, C., Tsai, R., and Massucco, A.}
\newblock Finite-difference least square methods for solving {Hamilton-Jacobi} equations using neural networks.
\newblock {\em Journal of Computational Physics 524\/} (2025), 113721.

\bibitem{guo2022monte}
{\sc Guo, L., Wu, H., Yu, X., and Zhou, T.}
\newblock Monte {{Carlo fPINNs}}: {{Deep}} learning method for forward and inverse problems involving high dimensional fractional partial differential equations.
\newblock {\em Computer Methods in Applied Mechanics and Engineering 400\/} (2022), 115523.

\bibitem{hao2025fractional}
{\sc Hao, Z., Cai, Z., and Zhang, Z.}
\newblock Fractional-order dependent radial basis functions meshless methods for the integral fractional laplacian.
\newblock {\em Computers \& Mathematics with Applications 178\/} (2025), 197--213.

\bibitem{he2025can}
{\sc He, R.~Y., Liang, Y., Zhao, H., and Zhong, Y.}
\newblock What can one expect when solving {PDE}s using shallow neural networks?
\newblock {\em arXiv preprint arXiv:2510.27658\/} (2025).

\bibitem{hnatiuk2025lazifying}
{\sc Hnatiuk, A., and Walter, D.}
\newblock Lazifying point insertion algorithms in spaces of measures.
\newblock {\em arXiv preprint arXiv:2508.03459\/} (2025).

\bibitem{huang2026adaptive}
{\sc Huang, J., Wu, H., and Zhou, T.}
\newblock Adaptive neural network basis methods for partial differential equations with low-regular solutions.
\newblock {\em Communications in Computational Physics 39}, 2 (2026), 553--577.

\bibitem{kansa1990multiquadrics}
{\sc Kansa, E.~J.}
\newblock Multiquadrics---{{A}} scattered data approximation scheme with applications to computational fluid-dynamics---{{II}} solutions to parabolic, hyperbolic and elliptic partial differential equations.
\newblock {\em Computers \& mathematics with applications 19}, 8-9 (1990), 147--161.

\bibitem{karnakov2024solving}
{\sc Karnakov, P., Litvinov, S., and Koumoutsakos, P.}
\newblock Solving inverse problems in physics by optimizing a discrete loss: Fast and accurate learning without neural networks.
\newblock {\em PNAS nexus 3}, 1 (2024), pgae005.

\bibitem{kharazmi2021hp}
{\sc Kharazmi, E., Zhang, Z., and Karniadakis, G.~E.}
\newblock Hp-{{VPINNs}}: {{Variational}} physics-informed neural networks with domain decomposition.
\newblock {\em Computer Methods in Applied Mechanics and Engineering 374\/} (2021), 113547.

\bibitem{krishnapriyan2021characterizing}
{\sc Krishnapriyan, A., Gholami, A., Zhe, S., Kirby, R., and Mahoney, M.~W.}
\newblock Characterizing possible failure modes in physics-informed neural networks.
\newblock {\em Advances in neural information processing systems 34\/} (2021), 26548--26560.

\bibitem{lagaris1998artificial}
{\sc Lagaris, I.~E., Likas, A., and Fotiadis, D.~I.}
\newblock Artificial neural networks for solving ordinary and partial differential equations.
\newblock {\em IEEE transactions on neural networks 9}, 5 (1998), 987--1000.

\bibitem{li2024entropy}
{\sc Li, Y., and Siegel, J.~W.}
\newblock Entropy-based convergence rates of greedy algorithms.
\newblock {\em Mathematical Models and Methods in Applied Sciences 34}, 05 (2024), 779--802.

\bibitem{li2024parameter}
{\sc Li, Z., Yang, S., and Wu, C.~J.}
\newblock Parameter inference based on {G}aussian processes informed by nonlinear partial differential equations.
\newblock {\em SIAM/ASA Journal on Uncertainty Quantification 12}, 3 (2024), 964--1004.

\bibitem{liao2024solving}
{\sc Liao, C.}
\newblock {Solving} partial differential equations with random feature models.
\newblock {\em arXiv preprint arXiv:2501.00288\/} (2024).

\bibitem{liu2021random}
{\sc Liu, F., Huang, X., Chen, Y., and Suykens, J.~A.}
\newblock Random features for kernel approximation: A survey on algorithms, theory, and beyond.
\newblock {\em IEEE Transactions on Pattern Analysis and Machine Intelligence 44}, 10 (2021), 7128--7148.

\bibitem{liu2022adaptive}
{\sc Liu, M., and Cai, Z.}
\newblock Adaptive two-layer {{ReLU}} neural network: {{II}}. {{Ritz}} approximation to elliptic {{PDEs}}.
\newblock {\em Computers \& Mathematics with Applications 113\/} (2022), 103--116.

\bibitem{liu2025integral}
{\sc Liu, X., Mao, T., and Xu, J.}
\newblock Integral representations of sobolev spaces via reluk activation function and optimal error estimates for linearized networks.
\newblock {\em arXiv preprint arXiv:2505.00351\/} (2025).

\bibitem{lu2021deepxde}
{\sc Lu, L., Meng, X., Mao, Z., and Karniadakis, G.~E.}
\newblock {{DeepXDE}}: {{A}} deep learning library for solving differential equations.
\newblock {\em SIAM review 63}, 1 (2021), 208--228.

\bibitem{malach2019deeper}
{\sc Malach, E., and Shalev-Shwartz, S.}
\newblock Is deeper better only when shallow is good?
\newblock {\em Advances in Neural Information Processing Systems 32\/} (2019).

\bibitem{meyer1992wavelets}
{\sc Meyer, Y.}
\newblock {\em Wavelets and operators}.
\newblock No.~37. Cambridge university press, 1992.

\bibitem{mhaskar2017and}
{\sc Mhaskar, H., Liao, Q., and Poggio, T.}
\newblock When and why are deep networks better than shallow ones?
\newblock In {\em Proceedings of the AAAI conference on artificial intelligence\/} (2017), vol.~31.

\bibitem{nelsen2025bilevel}
{\sc Nelsen, N.~H., Owhadi, H., Stuart, A.~M., Yang, X., and Zou, Z.}
\newblock Bilevel optimization for learning hyperparameters: {A}pplication to solving {PDEs} and inverse problems with {G}aussian processes.
\newblock {\em arXiv preprint arXiv:2510.05568\/} (2025).

\bibitem{pang2020npinns}
{\sc Pang, G., D'Elia, M., Parks, M., and Karniadakis, G.~E.}
\newblock n{PINN}s: Nonlocal physics-informed neural networks for a parametrized nonlocal universal laplacian operator. algorithms and applications.
\newblock {\em Journal of Computational Physics 422\/} (2020), 109760.

\bibitem{pang2019fpinns}
{\sc Pang, G., Lu, L., and Karniadakis, G.~E.}
\newblock f{PINN}s: Fractional physics-informed neural networks.
\newblock {\em SIAM Journal on Scientific Computing 41}, 4 (2019), A2603--A2626.

\bibitem{pieper2022nonconvex}
{\sc Pieper, K., and Petrosyan, A.}
\newblock Nonconvex regularization for sparse neural networks.
\newblock {\em Applied and Computational Harmonic Analysis 61\/} (2022), 25--56.

\bibitem{pieper2021linear}
{\sc Pieper, K., and Walter, D.}
\newblock Linear convergence of accelerated conditional gradient algorithms in spaces of measures.
\newblock {\em ESAIM: Control, Optimisation and Calculus of Variations 27\/} (2021), 38.

\bibitem{pieper2024nonuniform}
{\sc Pieper, K., Zhang, Z., and Zhang, G.}
\newblock Nonuniform random feature models using derivative information.
\newblock {\em arXiv preprint arXiv:2410.02132\/} (2024).

\bibitem{rahaman2019spectral}
{\sc Rahaman, N., Baratin, A., Arpit, D., Draxler, F., Lin, M., Hamprecht, F., Bengio, Y., and Courville, A.}
\newblock On the spectral bias of neural networks.
\newblock In {\em International conference on machine learning\/} (2019), PMLR, pp.~5301--5310.

\bibitem{RahimiRecht:08}
{\sc Rahimi, A., and Recht, B.}
\newblock Weighted sums of random kitchen sinks: replacing minimization with randomization in learning.
\newblock In {\em Proceedings of the 21st International Conference on Neural Information Processing Systems\/} (Red Hook, NY, USA, 2008), NIPS'08, Curran Associates Inc., p.~1313–1320.

\bibitem{raissi2019physics}
{\sc Raissi, M., Perdikaris, P., and Karniadakis, G.~E.}
\newblock Physics-informed neural networks: {{A}} deep learning framework for solving forward and inverse problems involving nonlinear partial differential equations.
\newblock {\em Journal of Computational physics 378\/} (2019), 686--707.

\bibitem{rathore2024challenges}
{\sc Rathore, P., Lei, W., Frangella, Z., Lu, L., and Udell, M.}
\newblock Challenges in training {PINNs}: A loss landscape perspective.
\newblock {\em arXiv preprint arXiv:2402.01868\/} (2024).

\bibitem{rosset2007}
{\sc Rosset, S., Swirszcz, G., Srebro, N., and Zhu, J.}
\newblock {$\ell_1$} regularization in infinite dimensional feature spaces.
\newblock In {\em Learning Theory: 20th Annual Conference on Learning Theory, {{COLT}} 2007, San Diego, {{CA}}, {{USA}}; June 13-15, 2007. {{Proceedings}} 20\/} (2007), Springer, pp.~544--558.

\bibitem{rychkov1999restrictions}
{\sc Rychkov, V.~S.}
\newblock On restrictions and extensions of the {B}esov and {Triebel--Lizorkin} spaces with respect to {L}ipschitz domains.
\newblock {\em Journal of the London Mathematical Society 60}, 1 (1999), 237--257.

\bibitem{saarinen1993ill}
{\sc Saarinen, S., Bramley, R., and Cybenko, G.}
\newblock Ill-conditioning in neural network training problems.
\newblock {\em SIAM Journal on Scientific Computing 14}, 3 (1993), 693--714.

\bibitem{sawano2018theory}
{\sc Sawano, Y.}
\newblock {\em Theory of Besov spaces}, vol.~56.
\newblock Springer, 2018.

\bibitem{schaback2006kernel}
{\sc Schaback, R., and Wendland, H.}
\newblock Kernel techniques: From machine learning to meshless methods.
\newblock {\em Acta numerica 15\/} (2006), 543--639.

\bibitem{shao2025solving}
{\sc Shao, Z., Pieper, K., and Tian, X.}
\newblock Solving nonlinear {PDEs} with sparse radial basis function networks.
\newblock {\em arXiv preprint arXiv:2505.07765\/} (2025).

\bibitem{Shin_2023}
{\sc Shin, Y., Zhang, Z., and Karniadakis, G.~E.}
\newblock Error estimates of residual minimization using neural networks for linear pdes.
\newblock {\em Journal of Machine Learning for Modeling and Computing 4}, 4 (2023), 73--101.

\bibitem{siegel2023greedy}
{\sc Siegel, J.~W., Hong, Q., Jin, X., Hao, W., and Xu, J.}
\newblock Greedy training algorithms for neural networks and applications to {PDEs}.
\newblock {\em Journal of Computational Physics 484\/} (2023), 112084.

\bibitem{siegel2022optimal}
{\sc Siegel, J.~W., and Xu, J.}
\newblock Optimal convergence rates for the orthogonal greedy algorithm.
\newblock {\em IEEE Transactions on Information Theory 68}, 5 (2022), 3354--3361.

\bibitem{siegel2023characterization}
{\sc Siegel, J.~W., and Xu, J.}
\newblock Characterization of the variation spaces corresponding to shallow neural networks.
\newblock {\em Constructive Approximation 57}, 3 (2023), 1109--1132.

\bibitem{siegel2024sharp}
{\sc Siegel, J.~W., and Xu, J.}
\newblock Sharp bounds on the approximation rates, metric entropy, and n-widths of shallow neural networks.
\newblock {\em Foundations of Computational Mathematics 24}, 2 (2024), 481--537.

\bibitem{sirignano2018dgm}
{\sc Sirignano, J., and Spiliopoulos, K.}
\newblock {DGM}: A deep learning algorithm for solving partial differential equations.
\newblock {\em Journal of computational physics 375\/} (2018), 1339--1364.

\bibitem{song2024does}
{\sc Song, C., Park, Y., and Kang, M.}
\newblock How does {PDE} order affect the convergence of {PINN}s?
\newblock {\em Advances in Neural Information Processing Systems 37\/} (2024), 73--131.

\bibitem{triebel2006}
{\sc Triebel, H.}
\newblock {\em Theory of Function Spaces III}, vol.~100 of {\em Monographs in Mathematics}.
\newblock Birkhäuser Basel, 2006.

\bibitem{wang2023expert}
{\sc Wang, S., Sankaran, S., Wang, H., and Perdikaris, P.}
\newblock An expert's guide to training physics-informed neural networks.
\newblock {\em arXiv preprint arXiv:2308.08468\/} (2023).

\bibitem{wang2021understanding}
{\sc Wang, S., Teng, Y., and Perdikaris, P.}
\newblock Understanding and mitigating gradient flow pathologies in physics-informed neural networks.
\newblock {\em SIAM Journal on Scientific Computing 43}, 5 (2021), A3055--A3081.

\bibitem{wang2022and}
{\sc Wang, S., Yu, X., and Perdikaris, P.}
\newblock When and why {PINNs} fail to train: A neural tangent kernel perspective.
\newblock {\em Journal of Computational Physics 449\/} (2022), 110768.

\bibitem{wendland2004scattered}
{\sc Wendland, H.}
\newblock {\em Scattered Data Approximation}, vol.~17.
\newblock Cambridge university press, 2004.

\bibitem{wojtowytsch2022representation}
{\sc Wojtowytsch, S., et~al.}
\newblock {Representation} formulas and pointwise properties for {Barron} functions.
\newblock {\em Calculus of Variations and Partial Differential Equations 61}, 2 (2022), 1--37.

\bibitem{xu2019frequency}
{\sc Xu, Z.-Q.~J., Zhang, Y., Luo, T., Xiao, Y., and Ma, Z.}
\newblock Frequency principle: Fourier analysis sheds light on deep neural networks.
\newblock {\em arXiv preprint arXiv:1901.06523\/} (2019).

\bibitem{xu2025overview}
{\sc Xu, Z.-Q.~J., Zhang, Y., and Zhou, Z.}
\newblock An overview of condensation phenomenon in deep learning.
\newblock {\em arXiv preprint arXiv:2504.09484\/} (2025).

\bibitem{ye2024fast}
{\sc Ye, Q., Tian, X., and Wang, D.}
\newblock A fast and accurate solver for the fractional fokker-planck equation with dirac-delta initial conditions.
\newblock {\em arXiv preprint arXiv:2407.15315\/} (2024).

\bibitem{zang2020weak}
{\sc Zang, Y., Bao, G., Ye, X., and Zhou, H.}
\newblock Weak adversarial networks for high-dimensional partial differential equations.
\newblock {\em Journal of Computational Physics 411\/} (2020), 109409.

\bibitem{zhang2025structured}
{\sc Zhang, S., Zhao, H., Zhong, Y., and Zhou, H.}
\newblock Structured and balanced multicomponent and multilayer neural networks.
\newblock {\em SIAM Journal on Scientific Computing 47}, 5 (2025), C1059--C1090.

\bibitem{zhang2023transnet}
{\sc Zhang, Z., Bao, F., Ju, L., and Zhang, G.}
\newblock Transferable neural networks for partial differential equations.
\newblock {\em Journal of Scientific Computing 99}, 1 (2024), 2.

\bibitem{zhuang2022radial}
{\sc Zhuang, Q., Heryudono, A., Zeng, F., and Zhang, Z.}
\newblock Radial basis methods for integral fractional laplacian using arbitrary radial functions.
\newblock {\em Available at SSRN 4283586\/} (2022).

\end{thebibliography}

\newpage

\begin{appendix}
\section{Proof of \Cref{thm:equivalence}}
\label{app:thmproof}

In this appendix, we show the proof of \Cref{thm:equivalence}.
The proof is based on the following three lemmas.

  \begin{lemma}
    Under the assumptions in Theorem \ref{thm:equivalence},
   $B^{s}_{1,1}(D)$ is continuously embeded in $\cV_\varphi^{\rm{meas}}(D)$.  
  \end{lemma}
  \begin{proof}
  The proof of this lemma follows by adapting the argument in 
    \cite[Theorem 24]{shao2025solving}. Let $u \in B^{s}_{1,1}(D)$. Since $D$ is a bounded Lipschitz domain, it follows from \cite[Theorems 1.105 and 1.118]{triebel2006} (see also \cite{rychkov1999restrictions}) that there exists a bounded linear extension operator $E \colon B^{s}_{1,1}(D) \to B^{s}_{1,1}(\R^d)$ such that $Eu|_D = u$, and
    \[
    \norm{Eu}_{B^{s}_{1,1}(\R^d)} \leq C \norm{u}_{B^{s}_{1,1}(D)}.
    \]
    Therefore, with a slight abuse of notation, we can assume without loss of generality that $u \in B_{1,1}^s(\R^d)$.
  Thus, it can be decomposed as (see, e.g., \cite[Section~2.1.1]{sawano2018theory}) 
\[
u = u_0 + \sum_{j=1}^\infty u_j
\text{ with }
u_0 = \psi(\sD)u, u_j = \zeta_j(\sD)u
\text{ and } \norm{u}_{B_{1,1}^s} = \norm{u_0}_{L^1} + \sum_{j=1}^\infty 2^{j\,s} \norm{u_j}_{L^1} < \infty.
\]
 Here, \(\psi \colon \R^d \to [0,1]\) is a function in \(C_c^\infty(\R^d)\) with
\[\
\chi_{B_1(0)}(\xi) \leq \psi(\xi) \leq \chi_{B_2(0)}(\xi)
\]
and \(\zeta_j = \psi_j - \psi_{j-1}\), where \(\psi_j(\xi) = \psi(2^{-j} \xi)\). Moreover, for any such function \(\psi\) (or \(\zeta_j\)) the symbol \(\psi(\sD)\) denotes the multiplication in Fourier space \(\psi(\sD) u = \cF^{-1}[\psi \cdot \cF[u]]\), where \((\psi \cdot \cF[u])(\xi) = \psi(\xi) \cF[u](\xi)\).

Next, we show that every \(u_j\) can be represented as a convolution with a kernel \(\phi_j(x) := 2^{jd}\phi(2^{j} |x|)\) with scale \(\sigma_j = 2^{-j}\) in the form
\begin{equation}
\label{eq:rec_mu}
u_j = \phi_j * \mu_j \quad\text{with } \mu_j \in L^1(\R^d).
\end{equation}
We argue this by defining the Fourier transform of \(\phi_j\) as
\(g_j = \cF\phi_j\) with \(g_j(\xi) = g_0(2^{-j}\xi)\) and
\[
\mu_j = [\psi_{j+1}/g_j](\sD) u_j.
\]
By Assumption \ref{assu:kernels}(2), \(g_j\) is bounded from below on $B_{r 2^j}(0)$ for some $r>0$. We may assume without loss of generality that \(r \geq 4\). 
Indeed, if not, we can always rescale the kernel \(\phi\) by a constant factor in the spatial variable, which does not change the function space \(\cV_\varphi^{\rm{meas}}(D)\) (up to equivalent norms).
Therefore \(g_j \in W^{k,1}(\R^d)\) for $k>d$ and is bounded from below on the support of \(\psi_{j+1}\). 
Thus the multiplier \(\psi_{j+1}/g_j\) is well-defined and is in \( W^{k,1}(\R^d) \), which implies that its inverse Fourier transform \(\varphi^{-1}_j = \cF^{-1}(\psi_{j+1}/g_j)\) is in \(L^1(\R^d)\). Due to \(\varphi^{-1}_j(x) = 2^{j d} \varphi^{-1}_j(2^j x)\) the \(L^1\) norm is independent of \(j\) and equal to a constant \(c\), which implies
\[
\norm{\mu_j}_{L^1} = \norm{[\psi_{j+1}/g_j](\sD) u_j}_{L^1}
= \norm{\varphi^{-1}_j * u_j}_{L^1}
\leq c\norm{u_j}_{L^1}.
\]
Moreover, due to \(g_j \psi_{j+1}/g_j * \zeta_j = \zeta_j\) and \(g_0 \psi_{1}/g_0 * \psi = \psi\), we have~\eqref{eq:rec_mu} as claimed.
Now, we define the measure
\[
\de\mu(y,\sigma) = \sum_{j=0}^\infty\sigma^{-s}_{j}\delta_{\sigma_j}(\sigma) \mu_j(y) \de y
= \sum_{j=0}^\infty 2^{j\,s}\delta_{2^{-j}}(\sigma) \mu_j(y) \de y.
\]
Without loss of generality, we assume $\sigma_{\max}\geq 1$.
We can then directly verify that \(u = \int_{\Omega} \varphi(\cdot; \omega) d\mu(\omega)\) and 
\[
\norm{u}_{\cV_\varphi^{\rm{meas}}(D)} \leq
\norm{\mu}_{M(\Omega)} = \sum_{j=0}^\infty 2^{j\,s} \norm{\mu_j}_{L^1}
\leq c \sum_{j=0}^\infty 2^{j\,s} \norm{u_j}_{L^1}
= c \norm{u}_{B^s_{1,1}(\R^d)}.
\qedhere
\]
\end{proof} 

    \begin{lemma}
        Under the assumptions in Theorem \ref{thm:equivalence},
   $\cV_\varphi^{\rm{meas}}(D)$ is continuously embeded in $\cV_\varphi^{\rm{atom}}(D)$. 
  \end{lemma}
  \begin{proof}
    Notice that the function $\varphi$ induces a map $\Phi: \Omega \to X$ defined by $\Phi(\omega) = \varphi(\cdot; \omega)$.
    Under Assumption \ref{assu:kernels} and $s > d(1-1/p)$, it is not hard to see that $\Phi: \Omega \to X$ is continuous,
    and the scalar-valued function $\|\Phi(\cdot)\|_X$ is integrable on $(\Omega, |\mu|)$
    for any \(\mu \in M(\Omega)\) since 
    \[
    \begin{split}
    &\int_{\Omega} \| \Phi(\omega)\|_X d|\mu|(\omega) = \int_{\Omega}  \sigma^{s-d} \left\|  \phi\left(\frac{|\cdot - y|}{\sigma}\right)\right\|_{L^p(D)} d|\mu|(\sigma, y)  \\
    \leq &  \int_{\Omega}  \left( \int_{\R^d} \left| \phi\left(\frac{|x - y|}{\sigma}\right)\right|^p dx \right)^{1/p} \sigma^{s-d} |\mu|(\sigma, y) \\
    \leq & \int_{\Omega}   \left( \int_{\R^d} \left| \phi\left(|x|\right)\right|^p dx \right)^{1/p} \sigma^{s-d + d/p} |\mu|(\sigma, y) \\
    \leq & \sigma_{\max}^{s-d + d/p}  \| \phi(|\cdot|) \|_{L^p(\R^d)} \cdot \| \mu \|_{M(\Omega)} \leq C.  \\
    \end{split}
    \]
    Here, we computed the case $1 \leq p < \infty$, and the case $p = \infty$ is similar.
    Therefore, $\Phi$ is Bochner integrable, i.e., $\Phi \in L^1(|\mu|; X)$.
    The rest of the proof is based on the standard approximation of Bochner integrable functions by simple functions.
    In particular, there exists $X$-valued simple functions 
$
    \Phi_{n} = \sum_{j=1}^{J_n}\Phi(\omega_{n, j})\chi_{A_{n, j}} 
$
with $\{A_{n, j}\}$ measurable and $\omega_{n, j} \in A_{n, j}$ such that
\[
    \int_{\Omega} \|\Phi(\omega) - \Phi_{n}(\omega)\|_{X} d\mu(\omega) \rightarrow 0.
\]
For any \(\mu \in M(\Omega)\), Let $\mu_{n} = \sum_{j=1}^{J_n} \mu(A_{n, j}) \delta_{\omega_{n, j}}$, then $\|\mu_n\|_{M(\Omega)} \leq \|\mu\|_{M(\Omega)}$ and 
\[
    \|u_{\mu} - u_{\mu_{n}}\|_{X} = \|\int_{\Omega}(\Phi - \Phi_n) d\mu\|_{X} \rightarrow 0.
\]
    Hence, any $u_\mu$ with $\| \mu \|_{M(\Omega)} \leq 1$ lies in $\overline{\cA_\varphi}^{\|\cdot\|_{X}}$, which implies $\cV_\varphi^{\rm{meas}}(D) \subseteq \cV_\varphi^{\rm{atom}}(D)$. 
    Also, for $u = u_\mu \in \cV_\varphi^{\rm{meas}}(D)$, by the lower semi-continuity of $\| \cdot \|_{\cV_\varphi^{\rm{atom}}(D)}$ in $X$ (since the sublevel sets are closed), we have
\[
    \|u_\mu\|_{\cV_\varphi^{\rm{atom}}(D)} \leq \liminf_{n\rightarrow \infty} \|u_{\mu_{n}}\|_{\cV_\varphi^{\rm{atom}}(D)} = \liminf_{n\rightarrow \infty}  \|\mu_{n}\|_{M(\Omega)} = \| \mu \|_{M(\Omega)}.
\]
Therefore, we have $ \|u_\mu\|_{\cV_\varphi^{\rm{atom}}(D)} \leq  \|u_\mu\|_{\cV_\varphi^{\rm{meas}}(D)}$ by selecting $\mu$ such that $\| \mu\|_{M(\Omega)} = \|u_\mu\|_{\cV_\varphi^{\rm{meas}}(D)}$. 
\end{proof}
    \begin{lemma}
        Under the assumptions in Theorem \ref{thm:equivalence},
  $\cV_\varphi^{\rm{atom}}(D)$ is continuously embeded in $B^{s}_{1,1}(D)$.
  \end{lemma}
    \begin{proof}
    
    First, notice that by the scaling \eqref{eq:kernelscaling}, for any $\omega = (y, \sigma)\in\Omega$
    \[
    \| \varphi(\cdot; \omega)\|_{B^{s}_{1,1}(\R^d)} \leq  C \| \phi(|\cdot|) \|_{B^{s}_{1,1}(\R^d)}, 
    \]
    where $C$ depends only on $s$ and $\sigma_{\max}$. Indeed, by $\varphi(\cdot; \omega) = \sigma^\gamma \phi(|\cdot - y| /\sigma)$, we have 
    \[
    \|\varphi(\cdot; \omega) \|_{L^1(\R^d)} = \sigma^{d+\gamma} \| \phi(|\cdot|)\|_{L^1(\R^d)} \leq \sigma_{\max}^s \| \phi(|\cdot|)\|_{L^1(\R^d)}. 
     \]
     By the integral form of the Besov norm (see \eqref{eq:BesovNormIntegralForm} with $D=\R^d$), 
     \[
     \begin{split}
     |\varphi(\cdot; \omega) |_{B^s_{1,1}(\R^d)} &= \sigma^\gamma\int_{0}^1 t^{-s-1} \sup_{|h|\leq t} \|\Delta^m_{h} \phi(|\cdot - y|/\sigma) \|_{L^1(\R^d)} dt  \\
     & =\sigma^\gamma \int_{0}^1 t^{-s-1} \sup_{|h|\leq t/\sigma} \sigma^d \|\Delta^m_{h} \phi(|\cdot|) \|_{L^1(\R^d)} dt  \\
     & = \sigma^{\gamma-s+d} \int_{0}^1 t^{-s-1} \sup_{|h|\leq t}  \|\Delta^m_{h} \phi(|\cdot|) \|_{L^1(\R^d)} dt  \\
     &= |\phi(|\cdot|)|_{B^s_{1,1}(\R^d)}.
      \end{split} 
     \]
     Therefore, the inequality holds.
     
    Using the above result, for all finite atomic sums $f = \sum_{n=1}^N a_n \varphi(x; \omega_n) \in \cA_\varphi$, 
     \[
    \| f\|_{B^{s}_{1,1}(D)} \leq  \| f \|_{B^{s}_{1,1}(\R^d)} \leq C \| \phi(|\cdot|) \|_{B^{s}_{1,1}(\R^d)} \sum_{n=1}^N |a_n| \leq C \| \phi(|\cdot|) \|_{B^{s}_{1,1}(\R^d)}.
     \]
   Since $D$ is a bounded Lipschitz domain and $s > d (1-1/p)$, we have the compact embedding of $B^{s}_{1,1}(D)$ into $X = L^p(D)$ (see, e.g., \cite[Theorem 1.107]{triebel2006}).
  One can then show that the above inequality also holds for all $f \in \overline{\cA_{\varphi}}^{\|\cdot\|_X} $. 
    More precisely, take $f \in \overline{\cA_{\varphi}}^{\|\cdot\|_X} $, then there exists a sequence $\{f_m\}_{m=1}^{\infty} \subseteq \cA_\varphi$ 
    such that $f_m \to f$ in $X$. Moreover, since $\{ f_m\}_{m=1}^\infty$ is a bounded set in $B^s_{1,1}(D)$ by the above inequality, 
    compact embedding implies the existence of a subsequence $\{f_{m_k}\}_{k=1}^{\infty}$ such that $f_{m_k} \to f$ in $X$. 
    Therefore, up to a subsequence (without relabeling), we have $f_{m_k}(x) \to f(x)$ for almost every $x \in D$.
    Using the standard integral form of the Besov norm and Fatou's lemma, we have
    \[
    \| f\|_{B^{s}_{1,1}(D)} \leq \liminf_{k\to\infty} \| f_{m_k}\|_{B^{s}_{1,1}(D)} \leq  C \| \phi(|\cdot|) \|_{B^{s}_{1,1}(\R^d)}.
    \]
    It is then straightforward to see that for all $f \in \cV_\varphi^{\rm{atom}}(D)$,
    \[
    \| f \|_{B^{s}_{1,1}(D)} \leq C \| \phi(|\cdot|) \|_{B^{s}_{1,1}(\R^d)} \| f \|_{\cV_\varphi^{\rm{atom}}(D)}.
    \]
    This completes the proof.
\end{proof}

    \section{Integer order chain rule}
    \label{app:int-order-chainrule}
We show that $L[\phi(\rho(\cdot;\omega))](x)$ can be computed analytically without relying on automatic differentiation for general integer-order differential operators. Without loss of generality, we consider a second-order linear differential operator \(L\) of the form
\[
L [u](x)=A(x):\nabla^{2}\phi(x)+B(x)\cdot\nabla u(x)+c(x)u(x),
\]
where \(A(x)\in\R^{d\times d}\), \(B(x)\in\R^{d}\), and \(c(x)\in\R\).
Let
\[
u(x):=\phi(\rho(x;\omega)),
\]
and define the unit direction
\[
e(x)=\frac{x-y}{|x-y|}
\]
A direct application of the chain rule yields
\[
\nabla u(x)=\frac{1}{\sigma}\phi'(\rho) e,\qquad \nabla^{2}u(x)
=
\frac{1}{\sigma^{2}}
\left[
\phi''(\rho)\,ee^{T}
+
\frac{\phi'(\rho)}{\rho}
\bigl(I-ee^{T}\bigr)
\right],
\]
where the expression is understood in the continuous extension sense at \(\rho=0\).
Substituting these expressions into the definition of \(L\), we obtain
\[
\begin{aligned}
L[\phi(\rho(\cdot;\omega))](x)
&=
\frac{1}{\sigma^{2}}
\Bigl(
\phi''(\rho)\,e^{T}A(x)e
+
\frac{\phi'(\rho)}{\rho}
\bigl(\tr A(x)-e^{T}A(x)e\bigr)
\Bigr) \\
&\quad
+
\frac{1}{\sigma}\phi'(\rho)\,B(x)\cdot e
+
c(x)\phi(\rho),
\end{aligned}
\]
which depends only on one-dimensional derivatives of the radial profile \(\phi\) and simple algebraic contractions with the coefficients \(A(x)\) and \(B(x)\).
This demonstrates that general second-order differential operators admit closed-form evaluation under the RBF ansatz, and higher integer-order operators follow analogously by repeated application of the chain rule.

    \section{Computation of fractional Laplacian}
    \label{app:frac}

Unlike classical integer-order Laplacian problems, where Dirichlet conditions are imposed locally on the boundary, fractional Laplacian problems involve exterior conditions defined nonlocally on $\Upsilon$. A commonly used approach, which we also adopt in our previous experiments, is to sample a large number of collocation points in $\Upsilon$ ($D^{c}$ in our case) to enforce the exterior condition. However, since $u_{c,\omega}$, in general, does not exactly match $g$ on $\Upsilon$ during training, this strategy may introduce additional approximation errors in the evaluation of the fractional Laplacian.

A more accurate, albeit more technical, approach for handling the exterior condition leverages the availability of the given Dirichlet condition $g$ and reformulates the problem using a corrected model function (\cite{burkardt2021unified,zhuang2022radial}) defined as
\begin{equation*}
        \tilde{u}_{c, \omega}(x) = \begin{cases}
u_{c, \omega}(x) &\quad  x \in D \\
g(x) &\quad x \in D^c 
\end{cases}.
    \end{equation*}
   Its fractional Laplacian then admits following expression
\begin{equation*}
   (-\Lap)^{\beta/2} \tilde{u}_{c, \omega}(x) = (-\Lap)^{\beta/2} u_{c, \omega} (x) + \cR_{g}[u_{c, \omega}](x),
\end{equation*}
where $(-\Lap)^{\beta/2} u_{c, \omega} (x)$ is given by \eqref{eq:fracLapRBF} and the residue $\cR_{g}$ is given by 
\begin{equation*}
    \cR_{g}[u](x)=  c_{d, \beta}\int_{D^{c}} \frac{u(z) - g(z)}{|x - z|^{d+\beta}} dz.
\end{equation*}
With this formulation, the extensive sampling of collocation points in $D^{c}$ is no longer needed. Nevetheless, while estimation of $\cR_{g}$ by numerical quadrature is not hard, it introduces challenges in computing gradients with respect to the inner weights and may cause numerical instability (see also \cite{zhuang2022radial}). Also, $\tilde{u}_{c, \omega}$ may not be continuous unless \(u=g\) is enforced on \(\partial D\), and thus $(-\Lap)^{\beta/2}\tilde{u}_{c, \omega}$ may not be rigorously defined. A detailed investigation of efficient and stable strategies for handling this method is left for future work.

\section{Implementation and Computation aspects}
\subsection{Implementation considerations for dynamically sized adaptive networks}
\label{app:imp}
A central practical challenge in adaptive, sparsity-promoting kernel methods is that the network width changes dynamically during optimization due to both insertion (Phase~I) and deletion (Phase~III). Unlike fixed-width neural networks, this leads to linear systems whose dimension varies across iterations, which is incompatible with just-in-time compilation (supported by JAX), static memory layouts, and batched linear algebra kernels commonly used in modern accelerator-based frameworks.

To address this issue, we adopt a padded active-set strategy that decouples the logical support size from the physical array dimensions. At each iteration, all parameters (outer weights, kernel centers, and shape parameters) are stored in arrays of a prescribed padding size, while a Boolean support mask tracks the currently active components. Gradient, Hessian, and proximal updates are restricted to the active set through masked operations, ensuring that inactive variables do not contribute to either the objective or the linearized subproblems.

At each Gauss–Newton iteration, we permute the system so that all currently active indices are moved to the front, yielding a compactified ordering. The linear system is then solved in a fixed-size padded representation using this reordered indexing. After computing the update, the solution is mapped back to the original parameter ordering by applying the inverse permutation. This strategy allows the active set to evolve dynamically while keeping array dimensions fixed, thereby avoiding repeated reallocation or reshaping of large matrices and ensuring consistency of the Gauss–Newton step within a static-shape computational framework.

Overall, this implementation enables dynamic network growth and shrinkage within a static-shape computational framework, balancing algorithmic flexibility with the practical constraints of high-performance JAX-based computation. The resulting solver remains fully adaptive while retaining stable memory usage, efficient linear algebra, and predictable compilation behavior.

\subsection{Computation overhead of Bi-Laplacian operator with automatic differentiation}
\label{app:compDiff}
We now examine the computational overhead of computing Bi-Laplacian operator for PINNs and our SparseRBFnet framework. The execution time of evaluation of bi-laplacian (with $d=4$) operator for MLP with auto-diff and SparseRBFnet with analytical expression is visualized in \Cref{fig:bilap_comp}.
\begin{figure}[t]
    \centering\includegraphics[width=0.7\linewidth]{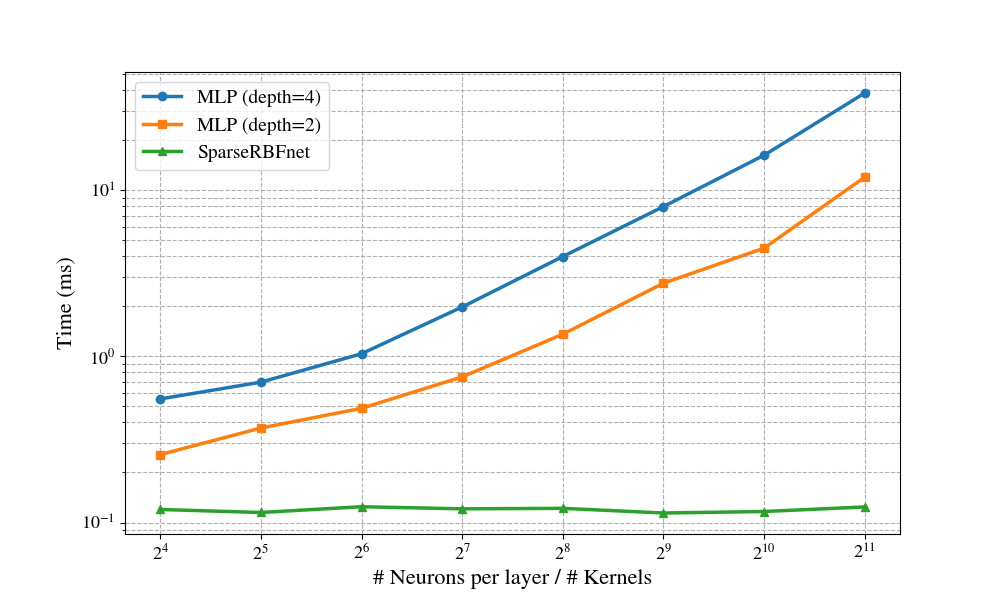}
    \caption{Comparison of execution time. All results are obtained by evaluating on 1,024 data points, and averaged from 10 individual runs.}
    \label{fig:bilap_comp}
\end{figure}

\end{appendix}

\end{document}